\documentclass[12pt,twoside,final,psamsfonts]{amsart}

\usepackage[psamsfonts]{amssymb}
\usepackage{times,a4wide}

\usepackage{color}		
\usepackage{epsfig}

\theoremstyle{plain}
\newtheorem*{theorem*}{Theorem}

\newtheorem{theorem}{Theorem}[section]
\newtheorem{proposition}[theorem]{Proposition}
\newtheorem*{proposition*}{Proposition}
\newtheorem{corollary}[theorem]{Corollary}
\newtheorem*{corollary*}{Corollary}
\newtheorem{lemma}[theorem]{Lemma}
\newtheorem*{lemma*}{Lemma}

\theoremstyle{definition}

\newtheorem{remarks}[theorem]{Remarks}
\newtheorem*{remark*}{Remark}
\newtheorem*{ack*}{Acknowledgements}
\newtheorem*{questions*}{Questions}
\newtheorem*{question*}{Question}
\newtheorem{question}{Question}

\theoremstyle{definition}
\newtheorem{definition}[theorem]{Definition}
\newtheorem*{definition*}{Definition}
\newcommand{\nc}{\newcommand}

\newcommand{\C}{{\mathbb C}}
\newcommand{\N}{{\mathbb N}}
\newcommand{\DD}{{\mathbb D}}
\newcommand{\R}{{\mathbb R}}

\newcommand{\T}{{\mathbb{T}}}

\newcommand{\mH}{{\mathcal H}}

\newcommand{\Int}{\operatorname{Int}}
\newcommand{\Alg}{\operatorname{Alg}}
\newcommand{\alg}{\operatorname{alg}}
\newcommand{\Hol}{\operatorname{Hol}}

\newcommand{\mult}{\operatorname{Mult}}
\nc{\Intl}{\Int_{l^{\infty}}}

\newcommand{\lra}{\longrightarrow}
\newcommand{\lmto}{\longmapsto}
\newcommand{\eps}{\varepsilon}
\newcommand{\vp}{\varphi}
\nc{\bea}{\begin{eqnarray}}
\nc{\eea}{\end{eqnarray}}
\nc{\beqa}{\begin{eqnarray*}}
\nc{\eeqa}{\end{eqnarray*}}
\nc{\Hi}{H^{\infty}}
\nc{\loi}{\ell^{\infty}}
\nc{\NL}{N^+\vert \Lambda}
\nc{\liL}{\lambda\in\Lambda}
\nc{\nn}{\nonumber}
\nc{\hf}{{\mathcal H}_{\Phi}}
\nc{\hF}{{\mathcal H}_{\Phi}}

\newenvironment{proof*}{\vskip 2mm\noindent {}}{$\blacksquare$ \vskip 2mm}
\numberwithin{equation}{section}

\nc{\card}{\operatorname{card}}
\nc{\tsn}{\tilde{\sigma}_n}
\nc{\tsk}{\tilde{\sigma}_k}
\nc{\tskp}{\tilde{\sigma}_k^+}
\nc{\tskm}{\tilde{\sigma}_k^-}
\nc{\D}{\displaystyle}
\nc{\vt}{\vartheta}

\title[Pointwise multipliers in Hardy-Orlicz
spaces, and interpolation]{Pointwise multipliers in Hardy-Orlicz spaces, 
and interpolation}

\author{Andreas Hartmann}

\address{Equipe d'Analyse \& G\'eom\'etrie,
Institut de Math\'ematiques de Bordeaux,
Universit\'e Bordeaux I, 351 cours de la Lib\'eration,
33405 Talence, France}
\email{Andreas.Hartmann@math.u-bordeaux1.fr}

\date{\today}



\begin{document}

\begin{abstract}
We study multipliers of Hardy-Orlicz spaces $\mH_{\Phi}$
which are strictly contained between $\bigcup_{p>0}H^p$ and 
so-called ``big'' Hardy-Orlicz spaces. 
Big Hardy-Orlicz spaces, carrying an algebraic structure,
are equal to their multiplier algebra,
whereas in classical Hardy spaces $H^p$, the multipliers reduce to
$H^{\infty}$. For Hardy-Orlicz spaces $\mH_{\Phi}$ 
between these two extremal situations and
subject to some conditions, we exhibit 
multipliers that are in Hardy-Orlicz spaces
the defining functions of which are related to 
$\Phi$. 
Even if the results do not entirely characterize
the multiplier algebra, some examples show that 
we are not very far from precise conditions.
In certain situations we see how 
the multiplier algebra grows in a sense from $\Hi$ to big Hardy-Orlicz
spaces when we go from classical $H^p$ spaces to big Hardy-Orlicz spaces.
However, the multiplier algebras are not always ordered as their
underlying Hardy-Orlicz spaces. Such an ordering holds 
in certain situations, but examples show that there are large Hardy-Orlicz
spaces for which the multipliers reduce to $\Hi$ so that the
multipliers do in general not conserve the ordering of the
underlying Hardy-Orlicz spaces.
We apply some of the multiplier results to construct 
Hardy-Orlicz spaces close to $\bigcup_{p>0}H^p$ and for which the
free interpolating sequences are no longer characterized by the
Carleson condition which is well known to characterize free interpolating 
sequences in $H^p$, $p>0$. 
\end{abstract}

\maketitle

\section{Introduction}

Let $\DD=\{z\in\C:|z|<1\}$ be the unit disk of the complex plane.
For a space of holomorphic functions on $\DD$, 
$X\subset\Hol(\DD)$, we define the multiplier algebra 
of $X$ by
\beqa
 \mult(X):=\{g\in\Hol(\DD):\forall f\in X, gf\in X\}.
\eeqa
We will consider spaces $X$ containing the constants so that automatically
$\mult(X)\subset X$.
Multiplier algebras have been studied in different settings.
They appear for instance in the context of cyclic functions
(see e.g.\ \cite{BS}).
Here we will rather be interested in
interpolation problems where multipliers come into
play for example via the Nevanlinna-Pick property
(see e.g.\ \cite{Ko}, \cite{MS}, \cite{Se}).
In this paper we will not consider the Nevanlinna-Pick
property but focus on spaces
for which the multiplier algebra is big in the sense
that its trace on $\Hi$-interpolating sequences 
contains more than only bounded sequences. 
(Recall that $\Hi$ is the space of bounded holomorphic functions
on $\DD$.) In such a situation it is possible
to interpolate bounded sequences on suitable non separated unions of
$\Hi$-interpolating
sequences. This was done in \cite{DSh} for Hardy spaces,
and a more general result can be derived from \cite{Ha} in
so-called (C)-stable spaces. Note (and this will be clear from
the definitions below) that if we can interpolate bounded sequences
by functions in the multiplier algebra then we can interpolate
freely in the initial space.

The spaces we are interested in here are included in the Smirnov class $N^+$.
Recall that the Nevanlinna class on $\DD$ is defined by
\beqa
 N=\{f\in\Hol(\DD):\sup_{0<r<1}\frac{1}{2\pi}\int_{\T}
 \log_+|f(re^{it})|dt<\infty\}.
\eeqa
Here $a_+=\max(0,a)$ for a real number $a$. 
It is well known that functions in the Nevanlinna class admit
non-tangential boundary values almost everywhere
on $\T=\partial\DD$. Then
\beqa
 N^+=\{f\in N:\sup_{0<r<1}\frac{1}{2\pi}\int_{\T}
 \log_+|f(re^{it})|dt=\frac{1}{2\pi}\int_{\T}
 \log_+|f(e^{it})|dt\}.
\eeqa
Hardy-Orlicz classes can then be defined by logarithmic convex functions
$\Phi=\vp\circ\log$ where $\vp$ is a positif, increasing, convex
function with $\vp(t)/t\to\infty$:
\beqa
 \mH_{\Phi}=\{f\in N^+:\int_{\T}\Phi(|f|)dm<\infty\}
\eeqa
(for more precise definitions, see Section \ref{S2}).
In the special situation when $\vp(t)=e^{pt}$ we obtain the
usual Hardy spaces, and when $\vp(t)=t^p$ we obtain so-called
big Hardy-Orlicz spaces. It is clear that in the first case
$\mult(H^p)=\Hi$ and in the second case $\mult(\mH_{\Phi})=\mH_{\Phi}$
since $\vp(t)=t^p$ satisfies a quasi-triangular inequality so
that $\mH_{\Phi}$ is an algebra and hence equal to its multiplier
algebra (see also \cite[Theorem 3.2]{HK}). 
A natural question arising from this observation
is to understand 
how the multiplier algebra changes from $\Hi$ for Hardy spaces $H^p$
(in a sense small Hardy-Orlicz spaces) to $\mH_{\Phi}$ for
big Hardy-Orlicz spaces.

Under certain conditions on the defining function $\vp$ of
the Hardy-Orlicz space under consideration $\mH_{\Phi}$ we will find
so-called admissible functions 
allowing the construction of new
Hardy-Orlicz spaces that are included (as well as the
algebras they generate) in the multipliers
of $\mH_{\Phi}$ ({\bf Theorem \ref{thm1}}), 
or that contain the multipliers of $\mH_{\Phi}$ 
({\bf Theorem \ref{thm3.2}}).
{\bf Corollary \ref{cor3.5}} shows that for certain scales
of Hardy-Orlicz spaces the gap between both inclusions is small.
{\bf Proposition \ref{propn3.4}} shows that Theorem \ref{thm1}
is optimal in a sense, and {\bf Proposition \ref{thm4.2}}
exhibits a function $g$ contained in
the space found in Theorem \ref{thm3.2} as an upper bound 
of the multiplier algebra of 
$\mH_{\Phi_{1/2}}$ (here $\Phi_{1/2}(t)=e^{\sqrt{t}}$) 
and not multiplying
on $\mH_{\Phi_{1/2}}$, 
thereby showing
that Theorem \ref{thm3.2} is not optimal. 

We will also discuss the ordering of the multiplier algebras.
Under some technical condition we prove in {\bf Proposition \ref{propn3.4}}
that the multiplier algebras conserve the ordering of their
underlying Hardy-Orlicz spaces. However, {\bf Theorem \ref{thm4.3}}
shows that this is not the general situation. Surprisingly
it turns out that
there are very big Hardy-Orlicz spaces for which the multipliers
reduce to $\Hi$. In particular there exist Hardy-Orlicz spaces
for which the ordering of the multipliers is in the opposite
direction with respect of the ordering of the initial Hardy-Orlicz
spaces.

Let us mention that 
multipliers of Hardy-Orlicz spaces have previously been considered
by Hasumi and Kataoka \cite{HK}, where conditions for $\Hi$
to contain or to be contained in the multiplier algebra are given,
and also by Deeb \cite{De85}. 
In \cite{HK} the authors also give some orderings of Hardy-Orlicz 
spaces that turn out to be useful in our situation.

The question of multipliers is strongly related in particular to {\it free}
interpolating sequences. Indeed, if we can interpolate
bounded sequences on a given sequence $\Lambda=\{\lambda_n\}_n\subset \DD$ 
by functions in the multiplier algebra, then $\Lambda$
is a free interpolating sequence (for this and the following
comments, precise definitions and results can be found
in Section \ref{S4}). Let us recall some facts on
interpolating sequences. 
It is well known that the Carleson condition 
$\inf_{\lambda} |B_{\Lambda\setminus\{\lambda\}}(\lambda)|>0$
characterizes free
interpolating sequences for $H^p$, $p\in (0,\infty]$, 
and Hardy-Orlicz spaces included in the scale of $H^p$ spaces
(see \cite{carl}, \cite{ShHSh}, \cite{Kab}, \cite{Har}).
We have already mentioned that in this situation $\mult(H^p)=\Hi$.
On the other hand, in $N$, $N^+$, and in big Hardy-Orlicz
spaces (e.g.\ $\vp(t)=t^p$), which are actually
algebras (and so equal to their multipliers), free interpolating sequences
are characterized by the existence of a harmonic majorant
of $\log(1/|B_{\Lambda\setminus\{\lambda\}}(\lambda)|)$
(see \cite{HMNT}, \cite{H06}). 
This condition is much weaker than the Carleson condition
(which can be restated as saying that 
$\log(1/|B_{\Lambda\setminus\{\lambda\}}(\lambda)|)$
admits in parti\-cular constants as harmonic majorants). For instance
separated sequences (with some conditions if we are in
big Hardy-Orlicz spaces) are interpolating in these classes.

Our starting point was to know whether there exist 
Hardy-Orlicz spaces beyond $\bigcup_{p>0}H^p$ for which the
Carleson condition still characterizes the interpolating sequences,
which leads us to the following question.

\begin{question}\label{q1}
Let $\mH_{\Phi}$ be a Hardy-Orlicz space. If the interpolating 
sequences of $\mH_{\Phi}$ are characterized by the Carleson
condition, is it true that $\mH_{\Phi}$ is included in the scale
$\bigcup_{p>0}H^p$?
\end{question}

In the light of this question, a first step is to construct examples of 
Hardy-Orlicz spaces above $\bigcup_{p>0}H^p$ which are
very close to the latter union and which have interpolating
sequences that are not Carleson.
The key to such a construction is the multiplier
algebra of the Hardy-Orlicz space under consideration
when this multiplier algebra is strictly bigger than $\Hi$.  
{\bf Corollary \ref{cor5.1}} exhibits multipliers of $\mH_{\Phi}$
where e.g.\ $\Phi(t)=t^{1/\log t}$ is in a sense very close to the
defining functions $t\lmto t^p$ of $H^p$, $p>0$.
In such a situation it is possible to use ideas of Douglas and
Shapiro \cite{DSh}
to interpolate bounded sequences on suitable
non separated unions of Carleson sequences. This yields
{\bf Corollary \ref{cor6.9}} which claims the existence 
of a non Carleson sequence 
which is  free interpolating for $\mH_{\Phi}$ when
$\mult(\mH_{\Phi})$ contains a Hardy-Orlicz space $\mH_{\Psi}$
that is strictly bigger than $\Hi$.

Since there exist large Hardy-Orlicz spaces for which the multipliers
reduce to $\Hi$ (see Theorem \ref{thm4.3}), we can give a more precise
version of Question \ref{q1}.

\begin{question}\label{q2}
If the multiplier algebra of a Hardy-Orlicz space containing strictly
$\bigcup_{p>0}H^p$ is equal to $\Hi$, does it have interpolating sequences
that are not Carleson?
\end{question}

The paper is organized as follows. In Section \ref{S2},
we will introduce the necessary material on Orlicz and 
Hardy-Orlicz spaces as well as some facts on decreasing
rearrangements.
The main results on multipliers are presented in Section \ref{S3}.
More precisely we exhibit Hardy-Orlicz spaces that bound
below and above the multiplier algebra of a given Hardy-Orlicz
space.
Orderings of multipliers will be discussed in Section \ref{sectord}.
Under some technical condition we will prove that the multiplier 
algebra inherits the ordering of the underlying Hardy-Orlicz spaces.
However we will prove that there are large Hardy-Orlicz spaces for
which the multipliers reduce to $\Hi$.
An important example is discussed in 
Section \ref{sectoptim} to show 
how far we are from
a characterization of the multiplier algebra. Other examples
of Hardy-Orlicz spaces coming very close to $\bigcup_{p>0}H^p$
and having unbounded multipliers
will be treated in Subsection \ref{sectexam}. These examples
are important 
in Section \ref{S4} where we
apply the multiplier 
results to the interpolation problem. 
Using ideas in the spirit of \cite{DSh} we will construct Hardy-Orlicz spaces 
$\mH_{\Phi}$
containing strictly $\bigcup_{p>0}H^p$ but being very close
to this union, and for which there exist non separated unions of
Carleson sequences which are interpolating for $\mH_{\Phi}$.

Finally a word concerning notation. For two expressions $u,v$
depending on the same discrete a continuous variable we will
sometimes write $u<<v$ if $u=o(v)$. As usual, $u\sim v$ means
that $u=v(1+\eps)$ (or $v=u(1+\eps)$) where $\eps=o(1)$. 

\begin{ack*}
Part of this work was presented at a joint
PICASSO-GDR AFHA meeting in Marseille. I would like to thank 
the participants of that meeting, in particular A.~Borichev 
and P.~Thomas, for some interesting questions that
are maybe answered in this paper.
\end{ack*}

\section{Orlicz and Hardy-Orlicz spaces}\label{S2}

When discussing Hardy-Orlicz spaces which are strictly
bigger than $\bigcup_{p>0}H^p$, one can consider
logarithmic convex defining functions. 
This is very natural since
convex functions conserve the subharmonicity of $\log|f|$
which makes it possible to define
Hardy-Orlicz spaces via the existence of harmonic majorants
(see \cite{RR}). For this reason we will consider
in all what follows defining
functions of the form $\vp\circ\log$ where 
$\vp:{\R}\lra [0,\infty)$ is a convex, nondecreasing
function with
$\lim_{t\to \infty} \vp (t)/t =\infty$.
To fix the ideas we should set $\vp(-\infty)=0$.
According to the terminology in \cite{Ru}
such a function is called {\it strongly convex}.

With such a function we will associate the Orlicz class 
on $\T$ defined by
\beqa 
 L_{\vp\circ\log}=L_{\vp\circ\log}(\T)
 =\{f \text{ measurable on }\T:\int_{\T} \vp(\log|f|)<\infty\}.
\eeqa
In order to simplify the notation, we will also write
\beqa
 \Phi=\vp\circ\log,
\eeqa
and so
\beqa
 L_{\Phi}=L_{\vp\circ\log}.
\eeqa
The functions $\vp$ or $\Phi$ are both called defining function
for the Orlicz class (hopefully no confusion will arise
in this paper). 

It should be noted that the Orlicz class is in general
not a vector space (see for instance the example 2 in
\cite[p.52]{RR} for the case of Hardy-Orlicz classes), 
and one can define two other spaces.
According to the notation in \cite{les}
we will call
\beqa
 L^{*}_{\Phi}:=\{f\text{ measurable on }\T:\exists a>0,
 \int_{\T}\Phi\left(\frac{|f|}{a}\right) <\infty\}
\eeqa
the Orlicz space, and
\beqa
 L^{\circ}_{\Phi}:=\{f\text{ measurable on }\T:\forall a>0,
 \int_{\T}\Phi\left(\frac{|f|}{a}\right) <\infty\}
\eeqa
the space of finite elements of $L^{*}_{\Phi}$. In \cite{LLQ}, the
latter space was called the Morse-Transue space.
Note that $L_{\Phi}^{\circ}\subset L_{\Phi}\subset L_{\Phi}^*$,
and in general these three classes are different.

In order to ensure that $L_{\Phi}$ is already a vector
space,
one sometimes adds another condition to that of
a defining function of an Orlicz space:
the function $\vp$ satisfies the $\tilde{\Delta}_2$-condition
if $\vp (t+2) \le M \vp (t) +K$,
$t \ge t_0$ for some constants $M,K \ge 0$ and $t_0\in \R$.
This condition is formulated in such a way that $\Phi$ 
satisfies the
usual $\Delta_2$-condition: there exist constants $M',K'\ge 0$ and
$s_0$ such that for all $s\ge s_0$ we have
\bea\label{Delta}
 \Phi(2s)\le M'\Phi(s)+K'.
\eea
If $\vp$ satisfies the $\tilde{\Delta}_2$-condition (or
$\Phi$ satisfies the $\Delta_2$-condition), then 
$L_{\Phi}=L_{\Phi}^{\circ}=L_{\Phi}^*$.

On $L^*_{\Phi}$ we can introduce the following functional. For 
$f\in L^*_{\Phi}$, let
\beqa
 \|f\|_{\Phi}:=\inf\{t>0:\int_{\T}\Phi\Big(\frac{|f|}{t}\Big) dm\le 1\}.
\eeqa
If $\Phi$ is convex, then 
$L^*_{\Phi}$ equipped with $\|\cdot\|_{\Phi}$
is a Banach space (also if we replace $\T$ by other measure
spaces), see \cite[p.120]{LT}.
The expression $\mathcal{J}_{\Phi}(f):=\int_{\T}\Phi(|f|)dt$ is
sometimes called a modular. It does of course 
not define a norm in general.

Here are some 
facts on orderings of Orlicz spaces. Let $\vp_1$ and $\vp_2$
be two strongly convex functions and set $\Phi_i=\vp\circ\log$,
$i=1,2$.
Then $\limsup_{t\to\infty}\Phi_1(t)/\Phi_2(t)
<\infty$ 
if and only if $L_{\Phi_2}\subset L_{\Phi_1}$ (see \cite[Theorem
1.3]{HK} where this result is proved for Hardy-Orlicz spaces,
but the argument works for Orlicz spaces).
The relation $L^*_{\Phi_1}\subset L^*_{\Phi_2}$
follows from $\lim_{t\to\infty}\Phi_1(t)/\Phi_2(kt)
=\infty$ for every $k>0$
(see \cite[Theorem 13.1]{KR} in case $\Phi_l$, $l=1,2$, convex).
Also, if two functions $\Phi_1$ and $\Phi_2$ (or $\vp_1$
and $\vp_2$) are comparable, i.e.\ there are constants $C_1$, $C_2$ 
with $C_1\vp_1(t)\le\vp_2(t)\le C_2\vp_1(t)$ for big $t$, 
then the corresponding (Hardy-)Orlicz spaces are 
equal. This allows for instance to replace the defining functions
by smooth ones. In all what follows we will thus suppose that 
the defining functions are sufficiently smooth.

It should be noted that it is possible to construct strongly convex
functions $\vp_1$ and $\vp_2$ for which $\limsup \vp_1(t)/\vp_2(t)
=+\infty$ and $\liminf \vp_1(t)/\vp_2(t)=0$. 
In such a situation, by the above cited result, 
no one of the considered Orlicz
spaces can be included in the other one.

As in the classical case of $L^p$-spaces, one can associate with 
$L_{\Phi}$ a subclass of boundary limits of a space
of holomorphic functions on the disk. Recall that $N^+$ is the
Smirnov class. 
The \emph{Hardy-Orlicz class} is defined as
\beqa
      \mH_{\Phi}=\mathcal{H}_{\vp\circ\log} 
 =\{f\in N^+:\int_\T\vp (\log |f(\zeta)|)\,d\sigma(\zeta)
      <\infty\}=N^+\cap L_{\Phi},
\eeqa
where $f(\zeta)$ is the non-tangential boundary value of $f$ at
$\zeta\in {\T}$, which exists almost everywhere since $f\in N^+$.
By \cite[Theorem 4.18]{RR} this definition is equivalent to
the definition via the existence of harmonic majorants that we
mentioned in the introduction to this section.
Also, since $\mH_{\Phi}$ as well as its 
multiplier algebra  
are contained in the Smirnov class $N^+$, 
we have
a factorization. 
Recall that each $f\in N^+$ can be
written as $f=IF$, where $I$ is an inner function and $F$ is 
outer in $N^+$. More precisely
\bea\label{outer}
 F(z):=[f](z)
 :=\exp\left(\int_{\T}\frac{\zeta+z}{\zeta-z}\log|f(\zeta)|dm(\zeta)
 \right),\quad z\in\DD,
\eea
and $\log|f|\in L^1(\T)$. If $f\in \mH_{\Phi}$
then $F\in \mH_{\Phi}$ and moreover $|f|\in L_{\Phi}(\T)$.

The classical examples are the following. When $\vp(t)=e^{pt}$ for
some $p>0$, then $\mathcal{H}_{\Phi}$ is simply the Hardy space
$H^p$, in which case $\mult(\mathcal{H}_{\Phi})$ is just
the algebra $H^{\infty}$ of bounded holomorphic functions on $\DD$.

The situation which has been considered in \cite{H06} in
connection with free interpolation is when
$\vp$ satisfies a quasi-triangular inequality:
\bea\label{deltaineq}
 \vp (a+b)\le c(\vp(a)+\vp(b)),
\eea
for some constant $c$, and all reals $a,b\ge t_0$, $t_0$ also fixed.
A simple example is $\vp(t)=t^p$.
The condition \eqref{deltaineq} 
is of course related to the $\Delta_2$ condition 
for $\vp$. In this situation, $\mH_{\Phi}$ is an algebra and
its multiplier algebra is of course the algebra itself:
$\mult(\mathcal{H}_{\Phi})=\mathcal{H}_{\Phi}$.

Analogously to the above definitions, 
we will write $\mathcal{H}_{\Phi}^*$ for the
Hardy-Orlicz space, $\mathcal{H}_{\Phi}^{\circ}$ for the
Hardy-Orlicz space of finite elements (or the Hardy-Morse-Transue
space). 
Again, if $\Phi$ satisfies the $\Delta_2$ condition than
all spaces are identical $\mH_{\Phi}=\mH_{\Phi}^*=\mH_{\Phi}^{\circ}$ 
and we simply write $\mathcal{H}_{\Phi}$.

We will introduce 
some conditions for a strongly
convex function $\vp$. 
Since we will consider multipliers, we are interested in 
the integrability of $\vp(\log|f|+\log|g|)$. Hence we would
like to know if we can add some growth to the argument $t$ of
$\vp$ without changing too much the growth of $\vp$. 
Here is a precise definition.

\begin{definition}\label{def2.1}
A convex, strictly increasing function $\vp:\R\lra \R^+$
with $\lim_{t\to\infty}\vp(t)/t=+\infty$ is said to satisfy
the {\it $\tilde{\Delta}$-condition}
if there is a $c>1$, $t_0 \in \R$ and a strictly increasing
concave function $\gamma: \R^+\lra \R^+$ with $\lim_{t\to \infty} 
\gamma(t)=\infty$ such that
for all $t\ge t_{\gamma}$
\bea\label{strDel}
 \frac{\vp(t+\gamma(t))}{\vp(t)}\le c.
\eea
A function $\gamma$ will be called {\it $\tilde{\Delta}$-admissible} if
\eqref{strDel} holds for suitable $c$ and $t_0$.
\end{definition}

The requirement of $\gamma$ being concave is not restrictive
since if an increasing function $\gamma$ satisfying \eqref{strDel}
exists, then we can replace it 
by a concave one.

This condition is stronger than the $\tilde{\Delta}_2$-condition
since instead of adding $2$ in the argument of $\vp$ we add a 
function that can tend to infinity. 
If $\vp$ itself already satisfies the
standard $\Delta_2$-condition \eqref{Delta}
(which leads us to big Hardy-Orlicz spaces), then
we can choose $\gamma(t)=t$ so that $\vp$ then
satisfies the $\tilde{\Delta}$-condition.

Our model case is
\beqa
 \vp_{\alpha}(t)=e^{t^{\alpha}}, \quad t\ge t_0>0,
\eeqa
where $\alpha\in (0,1)$. 
In this case we can construct the optimal function $\gamma$:
in order to have $\vp_{\alpha}(t+\gamma(t))\le c\vp_{\alpha}(t)$ it
is necessary and sufficient that $t\lmto
(t+\gamma(t))^{\alpha}-t^{\alpha}$ is bounded
(observe that necessarily $\gamma(t)\le t$). By standard
calculus, this is equivalent to
\beqa
 t^{\alpha}\left( \alpha\frac{\gamma(t)}{t}+o\left(\frac{\gamma(t)}
 {t}\right)\right)\le c, \quad t\ge t_{\gamma},
\eeqa
which happens if and only if
\beqa
 \gamma(t)\le Ct^{1-\alpha}.
\eeqa
So, we can choose $\gamma_{\alpha,C}(t):=Ct^{1-\alpha}$ which meets the
requirements of the function $\gamma$ in the definition
of the $\tilde{\Delta}$-condition above, and no
$\tilde{\Delta}$ admissible function can grow faster than any
$\gamma_{\alpha,C}$. 

Note that the $\tilde{\Delta}$-condition imposes a restriction on the
growth of $\vp$: clearly we cannot reach the function $\vp(t)=e^t$
(defining $H^1$), which is natural in view of our results.

When $\vp$ satisfies the $\tilde{\Delta}$-condition,
we will see (Theorem \ref{thm1}) that the admissible functions
$\gamma$ allow us to construct subalgebras of multipliers,
i.e.\ algebras which bound the multipliers of $\mH_{\Phi}$ from below.
So it is natural to ask
whether something sensible can be said about the
multipliers when condition \eqref{strDel} is not satisfied.
Actually, it turns out that if $\gamma$ is not admissible
then the algebras constructed in Theorem \ref{thm1} do no longer
bound the multiplier algebra from below. However
it seems too ambitious to hope for an upper bound
in this situation. Still, under some
mild growth condition on the quotient $\vp(t+\gamma(t))/\vp(t)$
we can obtain such an upper bound.

\begin{definition}
A convex, strictly increasing function $\vp:\R\lra \R^+$
with $\lim_{t\to\infty}\vp(t)/t=+\infty$ is said to satisfy
the {\it $\tilde{\nabla}$-condition}
if there is 
a strictly increasing
concave function $\gamma: \R^+\lra \R^+$ with $\lim_{s\to \infty} 
\gamma(s)=\infty$ and an $\eps>0$ such that
for all $s\ge s_{\gamma}$
\bea\label{strNab}
 \frac{\vp(s+\gamma(s))}{\vp(s)}\ge \log^{1+\eps} \vp(s).
\eea
A function $\gamma$ will be called {\it $\tilde{\nabla}$-admissible} if
\eqref{strNab} holds for suitable $s_{\gamma}$ and $\eps>0$.
\end{definition}

Let us discuss the $\tilde{\nabla}$-admissible functions for the model case
$\vp_{\alpha}(s)=e^{s^{\alpha}}$. The condition \eqref{strNab} is
equivalent to
\beqa
 e^{s^{\alpha}((1+\gamma/s)^{\alpha}-1)}\ge s^{(1+\eps)\alpha},\quad s\ge 
 s_{\gamma},
\eeqa
so that for example
\beqa
 \gamma(s):=\gamma_{\alpha,\eta}^{(\log)}(s):=
 (1+\eta) s^{1-\alpha}\log s,\quad s\ge s_{\gamma},
\eeqa
with $\eta>0$ works. 
Of course for ``bigger'' functions $\gamma$ 
the estimate in \eqref{strNab} is more easily true.
However, as we will see later on,
we will use reciprocals of $\tilde{\nabla}$-admissible functions
to find upper bounds for the multipliers. Hence we will get more
precise bounds with small $\tilde{\nabla}$-admissible 
functions $\gamma$. The reader may check that the
function $\gamma_{\alpha,\eta}^{(\log)}$ is not
$\tilde{\nabla}$-admissible for $\eta=0$.

\subsection{Decreasing rearrangements}\label{S2.1}

We will need some  facts on decreasing rearrangements
(for the material of
this subsection see for instance \cite[pp 114-120]{LT}).
Let us begin by recalling some basic facts.

Let $(\Omega,\Sigma,\mu)$ be a measure space (we will 
only be concerned with $\T$ equipped with the usual normalized
Lebesgue measure on Borel sets). With a measurable
function $f$ on $\Omega$ one associates the distribution function
\beqa
 \mu_f(t)=\mu\{\omega\in\Omega:|f(\omega)|>t\},\quad t>0,
\eeqa
and the decreasing rearrangement 
\beqa
 f^*(s)=\inf\{t>0:\mu_f(t)\le s\},\quad s\in (0,\mu(\Omega)).
\eeqa
Note that the decreasing rearrangement of $f$ is a positive 
function.
The main consequence on rearrangement invariant spaces
that we will use in the context of (Hardy-)Orlicz spaces is that
\bea\label{rearrinv}
 \int_{\T}\Phi(|f(t)|)dt=\int_0^1\Phi(f^*(t))dt.
\eea
(We have used here that $(\Phi\circ |f|)^*=\Phi\circ f^*$
since $\Phi$ is increasing.) We will also use the fact
that when $\Phi$ is convexe, then $L^*_{\Phi}$ is rearrangement
invariant \cite[p.120]{LT}.

The reader should notice that the initial measure space we
are interested in, i.e.\ $\T$ equipped with the Lebesgue measure,
can be identified with the measure space $[0,1]$ (equipped
with normalized Lebesgue measure) on which 
the decreasing rearrangement $f^*$ is defined. Thus
$f^*$ is 
obtained from $|f|$ by a measure preserving mapping $\alpha$
from $\Omega:=\T$ (i.e.\ $\Omega:=[0,1]$)
onto itself, 
so that $f^*(t)=|f(\alpha(t))|$. 

Moreover, it is clear that if 
a function $g$ multiplies on $\mH_{\Phi}$ then so does its
outer part (the modular $\mathcal{J}_{\Phi}$
in Hardy(-Orlicz) spaces does not ``feel''
the inner part). In all what follows we will thus assume
that the multiplier is outer (it could even be assumed
that $|g|\ge 1$). Let $g^*$ be the 
decreasing rearrangement of a multiplier $g$, and let $\alpha_g$ be
a corresponding measure preserving mapping of $\T$
(or $[0,1]$) onto itself. 
We have already mentioned that
$g$ is automatically in $\mH_{\Phi}$ and so $|g|\in L_{\Phi}$. By
\eqref{rearrinv} 
the function
$g^*$ is also in $L_{\Phi}$, and so we can associate
with it the outer function $G$ in $\mH_{\Phi}$ 
such that $|G|=g^*$ a.e.\ on $\T$.

\begin{lemma}
If $g\in\mult(\mH_{\Phi})$ then the outer function $G$
defined by $|G|=g^*$ a.e.\ $\T$ is also a multiplier on $\mH_{\Phi}$.
\end{lemma}

More generally it can be said that for every outer multiplier $g$ 
and every measure preserving mapping $\alpha:\T\to\T$, the
outer function $g_{\alpha}$ with $|g_{\alpha}|=|g\circ\alpha|$
a.e.\ on $\T$ is also a multiplier.

\begin{proof}
Let $\alpha$ be the measure preserving mapping
such that $g^*=|g\circ\alpha|$ a.e.\ on $\T$.
Let $f\in \mH_{\Phi}$ with outer part $F$. Then
the outer function with modulus $|F\circ\alpha^{-1}|$ is
also in $\mH_{\Phi}$ (with same modular $\mathcal{J}_{\Phi}$ as $f$), and
\beqa
 \int_{\T}\Phi(|f(\zeta)G(\zeta)|)dm(\zeta)
 &=&\int_{\T}\Phi(|f(\zeta)g(\alpha(\zeta))|)dm(\zeta)
 =\int_{\T}\Phi(|f\circ\alpha^{-1}(\zeta)g(\zeta)|)dm(\zeta)\\
 &=&\int_{\T}\Phi(|F(\zeta)g(\zeta)|)dm(\zeta)<\infty
\eeqa
\end{proof}
In the later discussions we can (and will) thus suppose that
the multiplier is outer, its only singularity is in 
$\zeta=1$, and $\theta\to |g(e^{i\theta})|$ is decreasing 
in $\theta$ on $(0,2\pi)$ ($2\pi$ corresponding to $1$).

\section{Multipliers - upper and lower bounds}\label{S3}

In this section 
we will give a general construction to obtain
multipliers of a Hardy-Orlicz space with a defining function $\vp$
satisfying the $\tilde{\Delta}$-condition. 
More precisely,
the $\tilde{\Delta}$-admissible functions $\gamma$ associated with $\vp$
allow the construction of defining functions
$\Psi_{\gamma}$ of Hardy-Orlicz spaces contained in the multiplier
algebra. 
Since $\mult(\mH_{\Phi})$ is an algebra it is clear that
when $\mH_{\Psi_{\gamma}}\subset \mult(\mH_{\Phi})$ then also 
$\Alg(\mH_{\Psi_{\gamma}})\subset \mult (\mH_{\Phi})$.
Here $\Alg(\mathcal{F})$ denotes the
algebra generated by a family of functions $\mathcal{F}$

Then,
using the $\tilde{\nabla}$-condition, we will 
give an inclusion of the multiplier algebra of $\mH_{\Phi}$
in another Hardy-Orlicz space the defining function of which 
$\Psi=\psi\circ\log$
is associated with $\tilde{\nabla}$-admissible functions.
Again, since $\mult(\mH_{\Phi})$ is an algebra, if it contains
$f\in \mH_{\Psi}$ then it contains also all powers $f^n$, $n\in\N$,
and so does $\mH_{\Psi}$.
Hence, setting $\Psi^{[n]}(t)=\psi(n\log t)$,
the inclusion $\mult(\mH_{\Phi})\subset \mH_{\Psi}$
implies that 
\beqa
 \mult(\mH_{\Phi})\subset\bigcap_{n\in \N^*}\mH_{\Psi^{[n]}}.
\eeqa

We will discuss both results in the model case $\vp(t)=\vp_{\alpha}(t)$.

Let us begin with a lower bound on the multiplier algebra.

\begin{theorem}\label{thm1}
Let $\vp$ be a strongly convex function satisfying the 
$\tilde{\Delta}$-condition and $\gamma$ a
$\tilde{\Delta}$-admissible function. 
Then 
\beqa
 \Alg(\mathcal{H}_{\Psi_{\gamma}})\subset \mult(\mathcal{H}_{\Phi})
\eeqa
where $\Psi_{\gamma}=\vp\circ\gamma^{-1}\circ\log$. 
\end{theorem}

\begin{remarks}
1) 
Obviously, $\mult(\mH_{\Psi})$ contains
the algebra generated by the
union over all $\mH_{\Psi_{\gamma}}$ where $\gamma$ is
admissible for $\vp$. 

2) In general $\psi_{\gamma}:=
\vp\circ\gamma^{-1}$ does not satisfy $\tilde{\Delta}_2$
and so we have to distinguish a priori
in the theorem between $\mH_{\Psi_{\gamma}}$,
$\mH^{\circ}_{\Psi_{\gamma}}$ and $\mH^*_{\Psi_{\gamma}}$.
This is of no harm since all these
spaces are of course included in $\Alg(\mH_{\Psi_{\gamma}})$ (and
we are of course interested in the biggest lower bound); see 
also some comments concerning the $\tilde{\Delta}_2$-condition
of $\psi_{\gamma}$ in the model case at the end of this section.
\end{remarks}

\begin{proof}
Let $f\in \mathcal{H}_{\Phi}$ and $g\in \mathcal{H}_{\Psi_{\gamma}}$.
Let $A:=\{\zeta\in\T:\log|g(\zeta)|\le \gamma(\log|f(\zeta)|)\}$ and
$A_0:=\{\zeta\in A:\log|f(\zeta)|\ge t_{\gamma}\}$.
Then
\beqa
 \int_{A_0}\vp(\log|gf|)dm
  &=&\int_{A_0}\vp(\log|f|+\log|g|)
 \le
 \int_{A_0}\vp\left(\log|f(\zeta)|+\gamma(\log|f(\zeta)|)\right) dm\\
 &\le&
 c\int_{A_0}\vp(\log|f(\zeta)|) dm,
\eeqa
and so the integral on $A_0$ converges.
Since on $A\setminus A_0$, $|f|$ and $|g|$ are bounded
(so that $\vp(\log|fg|)$ is bounded), the integral
also converges on $A$.

We will now consider the part of the integral on $B:=\T\setminus A$.
Set $B_0:=\{\zeta\in B:\gamma^{-1}(\log|g(\zeta)|)\ge t_{\gamma}\}$.
Clearly $\log|f(\zeta)|<\gamma^{-1}(\log|g(\zeta)|)$ on $B$.
Hence
\beqa
 \int_{B_0} \vp(\log|f(\zeta)|+\log|g(\zeta)|)dm
 &\le&\int_{B_0}\vp(\underbrace{\gamma^{-1}(\log|g(\zeta)|)}_{x}
   +\log|g(\zeta)|) dm\\
 &\le& c\int_{B_0}\vp(\gamma^{-1}(\log|g(\zeta)|))dm.
\eeqa
Since by assumption $g\in \mathcal{H}_{\Psi}$ where 
$\Psi=\vp\circ\gamma^{-1}\circ\log$, the last integral converges.
Since on $B\setminus B_0$ the functions $|f|$ and $|g|$ are bounded,
the integral converges also on $B$.
\end{proof}

Note that if $\vp$ satisfies the $\Delta_2$-condition (the case of
big Hardy-Orlicz spaces), then, as we have already mentioned, we can 
choose $\gamma(t)=t$. Hence
$\Psi_{\gamma}(t)=\vp\circ\gamma^{-1}\circ\log(t)=\vp(\log(t))=
\Phi(t)$, which confirms that we are in the algebra situation.

In order to show that Theorem \ref{thm1} is sharp we shall prove
that if a function $\gamma$ is not admissible for $\vp$, then
$\mH_{\Psi_{\gamma}}$ contains functions that do not multiply on
$\mH_{\Phi}$. Recall that $\Psi_{\gamma}=\vp\circ\gamma^{-1}\circ\log$.

\begin{proposition}\label{propn3.4}
Let $\vp$ be a strongly convex function, and 
let $\gamma$ be a concave function on $\R$ strictly increasing to
infinity such that 
\beqa
 \limsup_{t\to\infty}\frac{\vp(t+\gamma(t))}{\vp(t)}=+\infty.
\eeqa
Then there exists $g\in\mH_{\Psi_{\gamma}}$ such that 
$g\notin\mult(\mH_{\Phi})$.
\end{proposition}

\begin{proof}
The proof follows some ideas of the proof of \cite[Theorem 1.3]{HK}.
By the hypotheses, there exists a sequence $(t_n)_n$ such that
\beqa
 \left\{\begin{array}{l}
 \frac{\D \vp(t_n+\gamma(t_n))}{\D \vp(t_n)}\ge n,\\
 \vp(t_n)\ge\frac{\D 2^n}{\D n^2}.
 \end{array} 
 \right.
\eeqa
Set $\eps_n=(n^2\vp(t_n))^{-1}$. Clearly $\eps_n\le 2^{-n}$ so that
there exists a sequence $(\sigma_n)_n$ of disjoint measurable
subsets of $\T$ with $|\sigma_n|=\eps_n$, where $|\cdot|$
denotes the Lebesgue measure. Let $f$ be the outer
function the modulus of which is equal to $\sum_n
e^{t_n}\chi_{\sigma_n}$ on $\bigcup_n\sigma_n$ and $1$ otherwise
($\chi_{E}$ is the characteristic function of a measurable
set $E$). Then $\int_{\T}\Phi(|f|)dm=\sum_n\vp(t_n)|\sigma_n|
=\sum_n\vp(t_n)\eps_n=\sum_n\frac{1}{n^2}<\infty$. Hence $f\in
\mH_{\Phi}$.

In the same way, we let $g$ be the outer function the modulus of
which takes the values $\sum_n e^{\gamma(t_n)}\chi_{\sigma_n}$
on $\bigcup_n\sigma_n$ and $1$ elsewhere.
Then $\int_{\T}\vp(\gamma^{-1}(\log|g|))dm=\sum_n\vp(t_n)|\sigma_n|
<\infty$ (the reader might have observed that this is equal
to $\int_{\T}\Phi(|f|)dm$).
Let us compute the modular of their product
\beqa
 \int_{\T}\Phi(|fg|)dm&=&\int_{\T}\vp(\log|f|+\log|g|)dm
 =\sum_n \vp(t_n+\gamma(t_n))|\sigma_n|\ge
 \sum_n n\vp(t_n)\eps_n\\
 &=&\sum\frac{1}{n}=+\infty.
\eeqa

\end{proof}

We shall discuss this proposition further on an example in 
Section \ref{sectoptim}.

The next result discusses an upper bound of the multiplier algebra
via $\tilde{\nabla}$-admissibility.

\begin{theorem}\label{thm3.2}
Let $\vp$ be a strongly convex function satisfying the
$\tilde{\nabla}$-condition and $\gamma$ a $\tilde{\nabla}$-admissible
function.
Then
\beqa
 \mult(\mH_{\Phi})\subset  \bigcap_{n\in\N^*}\mH_{\Psi_{\gamma}^{[n]}},
\eeqa
where, as before, $\Psi_{\gamma}=\vp\circ\gamma^{-1}\circ\log$,
and $\Psi_{\gamma}^{[n]}(t)=\vp(\gamma^{-1}(n\log t))$.
\end{theorem}

Before proving the theorem, 
we will cite the following well-known property.

\begin{lemma}\label{corlem3.4}
Every positive decreasing function
on $(0,1]$ which is integrable on $(0,1)$ is ne\-cessarily
bounded by the function $t\lmto 1/t$ on $(0,t_0)$ for a
suitable $t_0\in (0,1)$.
\end{lemma}

\begin{proof}[Proof of Theorem \ref{thm3.2}]
Let $g\in \mult(\mH_{\Phi})$. 
As before we will suppose $g$ outer and $|g|$ equal to its
decreasing rearrangement. This will allows us to test $g$ 
against functions in $\mH_{\Phi}$
that approach the maximal possible growth of the
class. Since we have
identified $\T$ with $[0,1]$, we will set
\bea\label{eqw}
 w(t)=\frac{e}{t\log^{1+\eta}\frac{e}{t}},\quad t\in (0,1],
\eea
where $\eta\in (0,\eps)$ is fixed ($\eps$ being the value
associated with the $\tilde{\nabla}$-admissible function $\gamma$).
It is clear that $w\in L^1$. Let $f$ be the outer function in $\mH_{\Phi}$
such that $\log|f(e^{2\pi it})|=\vp^{-1}(w(t))$ a.e.\ on $(0,1]$.

Since $g$ is a multiplier on $\mH_{\Phi}$ and $|f|$, $|g|$ are
decreasing on $(0,1]$, the function
\beqa
 \phi:=\vp(\log|fg|)=\vp(\log|f|+\log|g|)
\eeqa
is decreasing  
on $(0,1]$ 
and integrable on this interval. 
By Lemma \ref{corlem3.4}, we get
\beqa
 \phi(t)\le \frac{e}{t},\quad t\in (0,t_0)
\eeqa
(where $t_0\in (0,1)$ is fixed suitably).

Hence 
\beqa
 \vp(\vp^{-1}(w)+\log|g|)\le\frac{e}{t},
\eeqa
and so
\beqa
 \log|g|\le \vp^{-1}(e/t)-\vp^{-1}(w).
\eeqa
Hence
\bea\label{eqestimlog}
 \Psi_{\gamma}(|g|)=\vp(\gamma^{-1}(\log|g|))
 \le \vp(\gamma^{-1}(\vp^{-1}(e/t)-\vp^{-1}(w))).
\eea
Since $\vp$ satisfies the $\tilde{\nabla}$-condition and
$\gamma$ is admissible we have 
\beqa
 \vp(s+\gamma(s))\ge \vp(s)\log^{1+\eps}\vp(s),\quad s\ge s_{\gamma},
\eeqa
so that
\beqa
 s\ge \gamma^{-1}(\vp^{-1}(\vp(s)\log^{1+\eps} \vp(s))-s),\quad s\ge 
 s_{\gamma}.
\eeqa
Applying $\vp$ to this inequality and 
choosing $s$ such that $w=w(t)=\vp(s)$ we obtain
\bea\label{eqestimwlogw}
 w\ge \vp(\gamma^{-1}(\vp^{-1}(w\log^{1+\eps}w)-\vp^{-1}(w))).
\eea
We will check that $1/t\le w\log^{1+\eps}w$. From \eqref{eqw}, we get
\beqa
 w(t)\log^{1+\eps}w(t)
 &=&\frac{e}{t\log^{1+\eta}(e/t)}\log^{1+\eps}\Big(\frac{e}{t
  \log^{1+\eta}(e/t)}\Big)\\
 &=&\frac{e}{t\log^{1+\eta}(e/t)}
  \Big(\log(e/t)-\log\log^{1+\eta}(e/t)\Big)^{1+\eps}\\
 &=&\frac{e}{t}\log^{\eps-\eta}(e/t)\Big(1-\frac{\log
   \log^{1+\eta}(e/t)}{\log^{1+\eta}(e/t)}\Big)\\
 &\ge&\frac{e}{t}
\eeqa
for $t$ sufficiently small since $\eps>\eta$. 
Injecting this into \eqref{eqestimwlogw}
we get 
\beqa
 w\ge \vp(\gamma^{-1}(\vp^{-1}(e/t)-\vp^{-1}(w))).
\eeqa
We recognize here the right hand side of \eqref{eqestimlog} so that
\beqa
 \Psi_{\gamma}(|g(e^{it})|)\le w(t)=\frac{1}{t\log^{1+\eps}(1/t)}.
\eeqa
Since $w\in L^1$, we conclude $g\in \mH_{\psi}$.
By the remarks in the introduction to this section we also have
$g\in \mH_{\Psi^{[n]}}$ for every $n\in\N^*$.
\end{proof}

{\bf Example.} Let us consider the model case $\vp_{\alpha}$,
$0<\alpha<1$. We have already constructed
the optimal function $\gamma_{\alpha,C}=Ct^{1-\alpha}$.
Obviously,
$\gamma_{\alpha}^{-1}(t)=(t/C)^{1/(1-\alpha)}$,
and $\psi_{\alpha,C}(t):=\vp_{\alpha}\circ \gamma_{\alpha,C}^{-1}(t)
=e^{dt^{\alpha/(1-\alpha)}}=\vp_{\alpha/(1-\alpha)}^d$, where 
$d=C^{-\alpha/(1-\alpha)}$. This together with Theorem \ref{thm1}
yields the first inclusion of the proposition below.

Note that $\Alg(\mH_{\Phi_{\alpha/(1-\alpha)}})=\bigcup_{d>0}
\mH_{\vp^d_{\alpha/(1-\alpha)}\circ\log}$ 
(one can
use that $fg=(1/2)\big( (f+g)^2-f^2-g^2\big)$ and $h\in \mH_{\alpha/
(1-\alpha)}$ implies $h^2\in \mH_{\vp^d_{\alpha/(1-\alpha)}}$ with
$d=(1/2)^{\alpha/(1-\alpha)}$).

For the second one we introduce another defining function. Set
\beqa
 \vp_{\alpha,\delta}^{(\log)}(t)=e^{{\delta}\left(
 \frac{t}{\log t}\right)^{\alpha/(1-\alpha)}},\quad t\ge t_0,
\eeqa
where $\delta>0$. 
Clearly, if $\beta<\alpha/(1-\alpha)$ then 
\beqa
 \vp_{\beta}(t)=e^{t^{\beta}}<<\vp_{\alpha,\delta}^{(\log)}(t),
 \quad t\to\infty.
\eeqa
Hence 
by the remarks on orderings of (Hardy-)Orlicz spaces
in Section \ref{S2}
\bea\label{HOincl}
 \mH_{\vp_{\alpha,\delta}^{(\log)}\circ\log}\subsetneqq 
 \mH_{\Phi_{\beta}}.
\eea

\begin{proposition}
Let $0<\alpha<1$.  
Then 
\beqa
 \Alg(\mH_{\Phi_{\alpha/(1-\alpha)}})
 =\bigcup_{d>0}\mH_{\vp_{\alpha/(1-\alpha)}^d\circ\log}\subset
 \mult(\mH_{\Phi_{\alpha}})
 \subset \bigcap_{\delta>0}\mH_{{\vp_{\alpha,\delta}^{(\log)}\circ\log}}.
\eeqa
\end{proposition}

Before proving this result, we give the following consequence
which is maybe easier to state and follows immediately from
this proposition and \eqref{HOincl}. 

\begin{corollary}\label{cor3.5}
Let $0<\alpha<1$ 
Then 
\beqa
 \Alg(\mH_{\Phi_{\alpha/(1-\alpha)}})
 =\bigcup_{d>0}\mH_{\vp_{\alpha/(1-\alpha)}^d\circ\log}\subset
 \mult(\mH_{\Phi_{\alpha}})
 \subset \bigcap_{0<\beta<\alpha/(1-\alpha)}\mH_{\Phi_{\beta}}.
\eeqa
\end{corollary}

Corollary \ref{cor3.5} shows that Theorem \ref{thm3.2} is optimal
in the sense that it allows 
to separate those Hardy-Orlicz spaces contained in the scale
$(\mH_{\Phi_{\alpha}})_{\alpha>0}$
and multiplying on $\mH_{\Phi_{1/2}}$ from those contained
in the scale that do not multiply on $\mH_{\Phi_{1/2}}$.
We could of course have replaced $\mH_{\Phi_{\beta}}$ by
$\bigcap_{n\in\N^*}\mH_{\Phi_{\beta}^{[n]}}$.

\begin{proof}[Proof of the proposition]
As already indicated, the first inclusion is established by the
above discussion. Let us consider the second inclusion.
Recall that for $\vp_{\alpha}$ the function
\beqa
 \gamma(s)=(1+\eta)s^{1-{\alpha}}\log s
\eeqa
is $\tilde{\nabla}$-admissible whenever $\eta>0$. 
Set $\Psi_{\gamma}=\vp\circ\gamma^{-1}\circ\log$.
It can be checked that
\beqa
 \gamma^{-1}(u)\sim \left(\frac{1-\alpha}{1+\eta}
 \frac{u}{\log u}\right)^{1/(1-\alpha)}, \quad u\to\infty.
\eeqa
So
\beqa
 \Psi_{\gamma}(t)=\exp\left[\left(\frac{1-\alpha}{1+\eta}
 \frac{\log t}{\log\log t}\right)^{\alpha/{1-\alpha}}(1+o(t))\right],
 \quad t\to\infty.
\eeqa
Since $\gamma$ is $\tilde{\nabla}$-admissible for  arbitrary $\eta>0$
and $o(t)$ is arbitrarily small, 
we can take
\beqa
 \Psi_{\gamma}(t)=\exp\left[(1-\delta)C_{\alpha}
 \left(\frac{\log t}{\log\log t}\right)^{\alpha/(1-\alpha)}\right].
\eeqa
where $\delta>0$ is arbitrary and $C_{\alpha}=(1-\alpha)^{\alpha/(1-\alpha)}$,
From Theorem \ref{thm3.2} we deduce that $\mult(\mH_{\Phi})
\subset\mH_{\Psi_{\gamma}}$. And by the general remarks we also
have $\mult(\mH_{\Phi})
\subset\mH_{\Psi_{\gamma}^{[n]}}$, where
\beqa
 \Psi_{\gamma}^{[n]}(t)=\exp\left[(1-\delta)C_{\alpha}
 \left(\frac{n\log t}{\log(n\log t)}\right)^{\alpha/(1-\alpha)}\right]
 =\exp\left[c
 \left(\frac{\log t}{\log n+\log\log t)}\right)^{\alpha/(1-\alpha)}\right]
\eeqa
with a suitable constant $c$. Clearly there exist
$\delta_1,\delta_2$ such that $\vp_{\alpha,\delta_1}^{(\log)}(\log t)
\le \Psi_{\gamma}^{[n]}(t)\le \vp_{\alpha,\delta_2}^{(\log)}(\log t)$
from which the remaining inclusion of the proposition 
follows. 
\end{proof}

The example $\vp_{\alpha}$ is quite instructive concerning
the behaviour of the multiplier algebra.
Clearly the index $\alpha/(1-\alpha)$ that we can associate
with $\vp_{\alpha}$ increases with 
$\alpha$ 
(we will see in Proposition \ref{prop3.1}
that for reasonable strongly convex functions --- and 
$\vp_{\alpha}$ are reasonable in our situation --- 
that the multiplier algebra increases with the space).
A crucial point is $\alpha=1/2$. Then $\psi_{1/2,1}(t)=
\vp_{1/2}\circ\gamma_{1/2,1}^{-1}(t)=\vp_1(t)=e^t$
which is the defining function for $H^1$, so that the
multiplier algebra of $\mH_{\Phi_{1/2}}$ contains $\Alg(H^1)
=\bigcup_{p>0}H^p$ (and it is contained in $\mH_{\Phi_{\beta}}$
for any $\beta<1$, and even in smaller Hardy-Orlicz spaces
defined by $\vp_{\alpha,\delta}^{(\log)}$). 

When $\alpha>1/2$, then by the corollary we have $\mult(\mH_{\alpha})
\subset \bigcap_{n\in\N^*}\mH_{\Phi^{[n]}_{1}}=\bigcap_{p>0}H^p$. 
Choosing $\beta\in (1,\alpha/(1-\alpha))$
we can even deduce that 
$\mult(\mH_{\alpha})\subset \mH_{\Phi_{\beta}}$ 
which is extremely small and close to $\Hi$.

Conversely, if $\alpha<1/2$, then since $\alpha/(1-\alpha)<1$,
we get $\vp_{\alpha/(1-\alpha)}(t)
=o(e^{pt})$ which yields $H^p\subset \mH_{\Psi_{\alpha/(1-\alpha)}}$ for
every $p>0$ and hence $\bigcup_{p>0}H^p\subset 
\mH_{\Psi_{\alpha/(1-\alpha)}}\subset \mult(\mH_{\Phi_{\alpha}})$.
So, in this case, the multiplier algebra is very big containing 
every $H^p$, $p>0$, and even bigger spaces.

Corollary \ref{cor3.5} tells us
that in this example the multiplier algebras
vary from very small spaces when $\mH_{\Phi}$ is close to
the classical Hardy spaces to very big ones when we approach
the big Hardy-Orlicz spaces.

Another observation can be made concerning the critical value $\alpha=1/2$.
For $\alpha\le 1/2$ the function $\xi:t\lmto (t+2)^{\alpha/(1-\alpha)}
-t^{\alpha/(1-\alpha)}$ is bounded so that $\psi_{\alpha,C}=\vp_{\alpha}
\circ\gamma_{\alpha,C}^{-1}$ satisfies that $\Delta_2$ condition, whereas
for $\alpha>1/2$ the function $\xi$ is unbounded and so
$\psi_{\alpha,C}\notin \Delta_2$.

A similar observation can be made in the context of Theorem \ref{thm3.2}.
By the above proof, the $\tilde{\nabla}$-admissible function $\gamma_{\alpha,
\eta}^{(\log)}$ satisfies $\big(\gamma_{\alpha,
\eta}^{(\log)}\big)^{-1}(t)\sim c (t/\log t)^{1/(1-\alpha)}$ for 
a suitable constant $c$. The function $\xi_{(\log)}:t\lmto
((t+2)/\log(t+2))^{\alpha/(1-\alpha)}-(t/\log t)^{\alpha/(1-\alpha)}$
is bounded if and only if $\alpha\le 1/2$ so that $\psi_{\gamma}
=\vp_{\alpha}\circ\big(\gamma_{\alpha,\eta}^{(\log)}\big)^{-1}$
satisfies the $\Delta_2$ condition if and only if $\alpha\le 1/2$.

\section{Orderings on multipliers}\label{sectord}

\subsection{A general result}

We begin the section with a general fact. Pick $\Phi_1=\vp_1\circ\log$
and $\Phi_2=\vp_2\circ\log$ two defining functions of Hardy-Orlicz
spaces, where $\vp_1,\vp_2$ are strongly convex. In Section \ref{S2}
we have mentioned that the condition 
\bea\label{phi1phi2}
 \limsup_{t\to\infty}\frac{\Phi_2(t)}{\Phi_1(t)}<+\infty 
\eea 
is equivalent 
to $\mH_{\Phi_1}\subset\mH_{\Phi_2}$. 
Replacing $\Phi_i$ by $\vp_i$ we get the same kind of estimate
for $\vp_2/\vp_1$ in \eqref{phi1phi2}. It is also possible to
replace moreover $\vp_1$ by $\vp_1+\vp_2$ without changing
$\mH_{\Phi_1}$, so that we can suppose that $h:=\vp_1-\vp_2$ is
strongly convex, and even that $h'=\vp_1'-\vp_2'$ tends to infinity 
at infinity.
This does unfortunately  not always imply
that $\vp^{-2}-\vp^{-1}$ is increasing.
However, if we assume the later to hold then the ordering of the
Hardy-Orlicz spaces is inherited by their respective multipliers.

\begin{proposition}\label{prop3.1}
Let $\vp_1,\vp_2$ be strongly convex functions.
If $\vp_2^{-1}-\vp_1^{-1}$ is increasing
then
\beqa
 \mult(\mH_{\Phi_1})\subset\mult(\mH_{\Phi_2}).
\eeqa
\end{proposition} 

\begin{proof}
We can suppose that $\vp_1$ and $\vp_2$ are differentiable.
By the hypothesis $\vp_2^{-1}-{\vp}_1^{-1}$ is
a strictly increasing function, so that $(\vp_2^{-1}-\vp_1^{-1})'\ge 0$.
Hence $((\vp_2^{-1})'(\vp_1(u))-(\vp_1^{-1})'
(\vp_1(u)))\vp_1'(u)\ge 0$ 
(note that obviously $\vp_1'\ge 0$). Hence
\beqa
 (\vp_2^{-1}\circ\vp_1)'(u)\ge 1,
\eeqa
for sufficiently big $u$. Setting $\xi(u):=\vp_2^{-1}\circ\vp_1\circ\log(u)$
we deduce from this that
$\xi'(u)\ge 1/u$ for big $u$.
Define now $\Xi=\Phi_2^{-1}\circ\Phi_1$. Then we get $(\log\circ\Xi)'(u)
=\xi'(u)\ge \frac{1}{u}$, and hence the function 
\beqa
 \Theta : t\lmto \frac{\Xi(t)}{t}=\frac{\Phi_2^{-1}\circ\Phi_1(t)}{t}
\eeqa
is increasing.

After these preliminary remarks let us come to the proof of the proposition.
Suppose $g\in\mult(\mH_{\Phi_1})$. Let $f\in\mH_{\Phi_2}$. We have
to check that $gf\in\mH_{\Phi_2}$. 
Define a measurable function on $\T$ by $f_0=\Phi_1^{-1}(\Phi_2(|f|))$.
Clearly there exists an outer function $F$ the modulus of which is
equal to $f_0$ almost everywhere on $\T$, and by construction
$F\in \mH_{\Phi_1}$. Since $g$ multiplies on $\mH_{\Phi_1}$ we
have $gF\in\mH_{\Phi_1}$. For the remaining argument we will suppose
$|g|\ge 1$ almost everywhere on $\T$ (we have already seen that $g$
can be supposed outer; it is also clear that $g$ is a mulitplier 
if and only if the outer function the modulus of which is equal
to $\max(1,|g|)$ is a multiplier). With this assumption we have
$|F|\le |gF|$ and since $\Theta$ is increasing we get
\beqa
 \frac{\Xi(|F|)}{|F|}\le\frac{\Xi(|gF|)}{|gF|},
\eeqa
i.e.
\beqa
 |g| \Phi_2^{-1}(\Phi_1(|F|))\le \Phi_2^{-1}(\Phi_1(|gF|)),
\eeqa
from where we get
\beqa
 \int_{\T}\Phi_2(|gf|)dm=\int_{\T}\Phi_2(|g|\Phi_2^{-1}(\Phi_1(|F|)))dm
 \le \int \Phi_1(|gF|)dm<\infty.
\eeqa
\end{proof}

Any ``reasonable'' pair of strongly convex functions
with $\mH_{\Phi_1}\subset \mH_{\Phi_2}$ satisfies the hypothesis
of Proposition \ref{propn3.4}. A simple example is
$\vp_1(t)=e^{t^{\alpha}}$
and $\vp_2(t)=e^{t^{\beta}}$ with $\alpha>\beta$ (this follows already
from Corollary \ref{cor3.5}). Another example is
given by $\vp_{1}(t)=e^{t}$ and
$\vp_2(t)=e^{t/\log t}$ for which it is simple to check that
$(\vp_2^{-1}-\vp_1^{-1})'\ge 0$.

A natural question raised by the preceding proposition is
whether there exist Hardy-Orlicz spaces for which the
ordering of the multipliers is in the opposite direction
of that of the Hardy-Orlicz spaces themselves.
The next subsection answers this question by
giving examples where the ordering of the multipliers
cannot be pulled back to the underlying Hardy-Orlicz spaces

\subsection{Small multipliers on large Hardy-Orlicz spaces}

Here we show that
there are large Hardy-Orlicz
spaces for which the multipliers reduce to $\Hi$, so that in
general the multipliers are not necessarily ordered as the
Hardy-Orlicz spaces (when these can be ordered).

Let us make more precise what we mean by ``large'' here. In fact
it turns out that the Hardy-Orlicz spaces we consider can be very
far from $\bigcup_{p>0}H^p$. We have to introduce a new class of
strongly convex functions. Set
\beqa
 \vp_{\alpha}(t)=e^{(\ln t)^{\alpha}},\quad t\ge t_0.
\eeqa
These functions define Hardy-Orlicz spaces $\mH_{\Phi_{\alpha}}$,
where $\Phi_{\alpha}=\vp_{\alpha}\circ\log$,
which are much bigger
than those associated with $\vp(t)=e^{t^{\alpha}}$
considered in Section \ref{S3}.
Let us observe that for every (concave) function $\gamma$
strictly increasing to infinity and such that $\gamma(t)=o(t)$
we have
\beqa
 \frac{\vp_{\alpha}(t+\gamma(t))}{\vp_{\alpha}(t)}
 &=&\exp\left[(\ln t)^{\alpha}\Big((1+\frac{\ln(1+\gamma(t)/t)}{\ln
     t})^{\alpha}-1\Big)\right]\\
 &=&\exp\left[\alpha\frac{\gamma(t)}{
     t(\ln t)^{1-\alpha}}+o(\frac{\gamma(t)}{
     t(\ln t)^{1-\alpha}})\right].
\eeqa
which is bounded when $\gamma(t)\le Ct (\ln t)^{\alpha-1}$.
The latter expression suggests that we could attain a growth
faster than the identity when $\alpha>1$.
In this situation the above
computations, which work under the
assumption $\gamma(t)=o(t)$, are of course false.
Anyway, since we
are only interested in concave $\gamma$ it is not worth while seeking
$\gamma$ growing faster than the identity.
So, $\vp_{\alpha}$ satisfies the $\tilde{\Delta}$-condition and
for instance $\gamma_p(t)=t^{p}$ is admissible for every $p\in (0,1)$.
Using Theorem \ref{thm1}
this implies that the multipliers of $\mH_{\Phi_{\alpha}}$ contain
a very big space: $\mH_{\Psi_{\alpha,p}}$, where
$\Psi_{\alpha,p}=\psi_{\alpha,p}\circ\log$, $\psi_{\alpha,p}(t)
=\vp_{\alpha}\circ\gamma_p^{-1}(t)=
e^{(1/p)^{\alpha}(\ln t)^{\alpha}}=\vp_{\alpha}^{(1/p)^{\alpha}}$.

We have the following result.

\begin{theorem}\label{thm4.3}
For every $\beta>1$ 
there exists a strongly convex function $\vp$ 
satisfying the $\tilde{\Delta}_2$-condition such that 
$\mH_{\Phi}$ contains $\mH_{\Phi_{\beta}}$ and  
\beqa
 \mult(\mH_{\Phi}) =\Hi.
\eeqa
\end{theorem}

\begin{proof}
We begin by constructing the strongly convex function on
$\R^+$. Suppose $\vp(1)=1$.
Let $(t_n)$ be a sequence of positive real numbers tending
strictly to infinity and $t_1=1$. We will also assume that $(t_{n+1}-t_n)$
goes to infinity. The construction of $\vp$ goes inductively.
On each interval $I_n=[t_n,t_{n+1})$ the function is
affine with $\vp(t_n)=\lim_{t\to t_n^-}\vp(t)$ so that $\vp$ is 
continuous in $t_n$ and with slope $\vp(t_n)$ (the function
doubles its values from $t_n$ to $t_n+1$). This yields of
course a convex function the slope of which tends to infinity
from where we deduce that it is strongly convex. 
(It is clear how to extend $\vp$ to $\R_-$.)

Let us check that by a suitable choice of $(t_n)$ we obtain
a function $\vp$ tending more slowly to infinity than 
$\vp_{\beta}$. This will show that $\mH_{\Phi_{\beta}}\subset \mH_{\Phi}$.
Fix $\gamma >\frac{\D 1}{\D \beta -1}$.
By construction $\vp(t_{n+1})=\vp(t_n)(1+(t_{n+1}-t_n))$.
Set $t_{n+1}=t_n+e^{n^{\gamma}}-1$, 
so that $\vp(t_{n+1})=e^{n^{\gamma}}\vp(t_n)$,
and an immediate induction yields $\vp(t_n)=e^{\sum_{k=1}^{n-1}
k^{\gamma}}\vp(t_1)$, where  $\vp(t_1)=1$. It is well known that
$\sum_{k=1}^{n-1}k^{\gamma}\sim\frac{(n-1)^{\gamma+1}}{\gamma+1}$, from
where we deduce that $e^{(1-\eps)(n-1)^{\gamma+1}/(\gamma+1)}
\le \vp(t_{n})\le e^{(1+\eps)(n-1)^{\gamma+1}/(\gamma+1)}$
for sufficiently big $n$ (depending on $\eps$).
Now $t_{n+1}-t_n\sim e^{n^{\gamma}}$, so that 
$t_{n}=t_1+\sum_{k=1}^{n-1}(t_{k+1}-t_k)
\sim\sum_{k=1}^{n-1} e^{k^{\gamma}}\ge e^{(n-1)^{\gamma}}$.
By assumption $\gamma>\frac{\D 1}{\D \beta-1}$, 
so that $\gamma\beta>\gamma+1$. Hence for sufficiently big $n$
\beqa
 \vp_{\beta}(t_{n-1})\ge
 e^{(\ln{{e^{(n-2)^{\gamma}}}})^{\beta}}
 =e^{(n-2)^{\gamma\beta}}
 >>e^{(1+\eps)(n-1)^{\gamma+1}/(\gamma+1)}
 \ge\vp(t_n).
\eeqa
This implies that on the whole interval $I_n$ the function
$\vp_{\beta}$ dominates $\vp$. Since this is true for
every interval $I_n$ ($n$ sufficiently big), we can 
deduce that $\mH_{\Phi_{\beta}}\subset \mH_{\Phi}$.

The remaining part of the proof is again built on the arguments of
of \cite[Theorem 1.3]{HK}.
Suppose now that there exists an unbounded multiplier $g$ for $\mH_{\Phi}$.
Let $\sigma_k=\{\zeta\in\T:\log|g(\zeta)|\in [k,k+1)\}$
which are of positive measure by assumption.
Since $\vp$ tends to infinity, there exists a subsequence
$(t_{n_k})_k$ such that $\vp(t_{n_k})\ge\frac{\D 1}{\D k^2 |\sigma_k|}$.
Then we can find $\sigma_k'\subset \sigma_k$ such that
$\vp(t_{n_k})|\sigma_k'|=\frac{\D 1}{\D k^2}$. Let $f$ be
the outer function the boundary values of which are in
modulus equal to $\sum_k e^{t_{n_k}}\chi_{\sigma_k'}$ on
$\bigcup_k\sigma'_k$ and $1$ elsewhere. Then
$\int_{\T}\Phi(|f|)dm=\sum_k\vp(t_{n_k})|\sigma'_k|
=\sum_k\frac{\D 1}{\D k^2}<\infty$.

On the other hand, since for $\gamma>0$ we have
$\vp(t_{n_k}+\gamma)\ge\vp(t_{n_k})+\vp(t_{n_k})\gamma\ge
\gamma\vp(t_{n_k})$, we obtain
\beqa
 \int_{\T}\Phi(|fg|)dm&=&\sum_k\int_{\sigma'_k}
 \vp(\log|f|+\log|g|)dm\ge \sum_k\int_{\sigma'_k}
 \log|g|\vp(t_{n_k})dm \\
 &\ge& \sum_k k\vp(t_{n_k})|\sigma'_k|
 =\sum_k\frac{1}{k}\\
 &=&\infty.
\eeqa
So, $g$ does not multiply $f$ to a function in $\mH_{\Phi}$. We
have reached a contradiction, and any multiplier in $\mH_{\Phi}$
has to be bounded.

It is easily checked that, by construction, $\vp$ satisfies 
the $\tilde{\Delta}_2$-condition, so that we also have
$\Hi\subset\mult(\mH_{\Phi})$.
\end{proof}

\section{Some more examples}\label{sectoptim}

\subsection{Optimality of the conditions}

We begin this section with an example discussing
the optimality of the results of Section \ref{S3}. 
We have already seen in Proposition 
\ref{propn3.4} that the result of Theorem \ref{thm1} is 
in a sense sharp: whenever a concave function $\gamma$
is not admissible for $\vp$ then we can find a function
in $\mH_{\Psi_{\gamma}}$, where $\Psi_{\gamma}=\vp\circ\gamma^{-1}
\circ\log$, that does not multiply on $\mH_{\Phi}$.

We will discuss this more thouroughly here
through the example $\vp_{1/2}(t)=e^{\sqrt{t}}$.
Recall that in this situation our
Theorem \ref{thm1} gave the inclusion $\bigcup_{p>0}H^p
\subset \mult \mH_{\Phi_{1/2}}$.
On the other side, Theorem \ref{thm3.2} shows that 
$\mult \mH_{\Phi_{1/2}}\subset
\mH_{\vp_{1/2,\delta}^{(\log)}\circ\log}$ for every $\delta>0$.
Recall that $\vp_{1/2,\delta}^{(\log)}(t)=e^{\delta
\frac{t}{\log t}}$.

Here we will use Proposition \ref{propn3.4}
to show the existence of a function $g$
not multiplying on $\mH_{\Phi_{1/2}}$ and which is in
Hardy-Orlicz classes coming much closer to $\bigcup_{p>0}H^p$ than
do the spaces $\mH_{\vp_{1/2,\delta}^{(\log)}\circ\log}$, $\delta>0$.
This shows that  Theorem \ref{thm3.2} is not optimal (even if 
Corollary \ref{cor3.5} gave us some optimality; see the
comments after that corollary).

We begin by introducing a new scale of Hardy-Orlicz spaces.
In order to simplify the
notation we will set for $k\ge 1$
\beqa
 \log_k:=\underbrace{\log\circ\cdots \circ\log}_{k\text{ times}}.
\eeqa
We will also set $e_1:=e$ and
$e_{k+1}:=e^{e_k}$.
Then for $k\ge 2$ 
we introduce the functions $\vp_{(k)}$ which are defined by
\beqa
 \vp_{(k)}(t)=\exp\Big(\frac{\D t}{\D \log_{k-1}(t)}\Big),
 \quad \text{for } t \ge e_{2k}.
\eeqa
The functions are completed suitably for $t<e_{2k}$ to convex 
functions.

The spaces $\mH_{\Phi_{(k)}}$, where $\Phi_{(k)}=
\vp_{(k)}\circ\log$, come extremely close to $\bigcup_{p>0}H^p$
when $k\to\infty$ without ever atteining the latter union.

\begin{proposition}\label{thm4.2}
There is a function $g\in\bigcap_{k\ge 1}\mH_{(k)}$ that
does not multiply on $\mH_{\Phi_{1/2}}$.
\end{proposition}

\begin{proof}
Using the numbers $e_k$, we will define a function $\gamma$
which is not admissible for $\vp_{1/2}$. Let $\eps:\R^+\to\R$ be
continuous and piecewise affine such that
\beqa
 \eps(e_k)=k,\quad k\ge 1.
\eeqa
The function $\eps$ is clearly concave on $[1,+\infty)$, and so
will be $\gamma$ defined by $\gamma(t)=\sqrt{t}\eps(t)$ on
$[1,+\infty)$. The function $\gamma$ is not admissible since
\beqa
 \frac{\vp_{1/2}(t+\gamma(t))}{\vp_{1/2}(t)}
 =e^{\sqrt{t+\sqrt{t}\eps(t)}-\sqrt{t}}
 =e^{\frac{1}{2}\eps(t)+o(\eps(t))}
\eeqa
tends to infinity (we had already mentioned in 
Section \ref{S2} that any $\tilde{\Delta}$-admissible function
for $\vp_{\alpha}$ can grow at most as $t\to Ct^{1-\alpha}$). 
Hence by Proposition \ref{propn3.4}
there is a function in $\mH_{\Psi_{\gamma}}$ that does not 
multiply on $\mH_{\Phi}$. 

We will show that 
\beqa
 \mH_{\Psi_{\gamma}}\subset \mH_{\Phi_{(k)}},
\eeqa
for every $k$. 
For this it is sufficient to check that for every $k\in\N^*$
there is a $t_k$ such that for every $t\ge t_k$
\beqa
 \vp_{1/2}\circ\gamma^{-1}(t)\ge e^{\frac{t}{\log_k t}}.
\eeqa
Passing to logarithms and observing that $\gamma$ is continuous and
strictly increasing to $+\infty$ 
so that we can change to the variable 
$u=\gamma^{-1}(t)$, we are led to the verification of
\beqa
 \log\vp_{1/2}(u)={\sqrt{u}}\ge \frac{\gamma(u)}{\log_k(\gamma(u))}
 =\frac{\sqrt{u}\eps(u)}{\log_k(\sqrt{u}\eps(u))}
\eeqa
for $u$ sufficiently big. This is of course equivalent to
$\log_k(\sqrt{u}\eps(u))\ge \eps(u)$ for big $u$. The
left hand side of this estimate behaves like $\log_k u$ so
that it remains to show that $\eps$ is neglectible with respect
to $\log_k$ at infinity. Fix such a $k$ and let
$n>k$. Then for $t\in [e_n,e_{n+1})$ we have
$\log_k(t)\ge  \log_k(e_n)=e_{n-k}$ which goes ``extremely'' fast
to infinity (one could observe 
that for $k\ge 1$
we have $e_{k+1}/e_k=e^{e_k}/e_k\ge M:=e^{e-1}$ since $e^t\ge Mt$ for
$t\ge e$, so that $e_{n-k}$ grows at least exponentially in $n$), 
whereas $\eps(t)\le \eps(e_{n+1})=n+1$.
\end{proof}

\subsection{Big multipliers in small 
Hardy-Orlicz spaces}\label{sectexam}

In this section
we will show that there are Hardy-Orlicz spaces
beyond $\bigcup_{p>0}H^p$ coming very close to 
$\bigcup_{p>0}H^p$ and containing unbounded
multipliers. More precisely, such Hardy-Orlicz spaces
contain  Hardy-Orlicz spaces strictly bigger than $\Hi$.
This is of central interest in the interpolation problem
since it will allow to conclude that such Hardy-Orlicz
spaces 
admit interpolating sequences which are not Carleson, i.e.\ 
which are not interpolating for $\Hi$.

The key result to our examples here is the following proposition.

\begin{proposition}\label{prop5.1}
Let $\vp$ be a strongly convex function on $\R$ 
strictly increasing to $+\infty$. Let $(t_n)_n$ be the
sequence defined by 
\beqa
 \vp(t_n)=2^n,\quad n\in\N.
\eeqa
If $(t_{n+1}-t_n)_n$ tends to infinity, then
$\vp$ is $\tilde{\Delta}$-admissible, i.e.\
there exists $\gamma:[t_0,+\infty)\to\R$ concave, increasing
with $\lim_{t\to\infty}\gamma(t)=+\infty$ such that
\bea\label{tDprop5.1}
 \vp(t+\gamma(t))\le 4\vp(t),\quad t\ge t_0.
\eea
\end{proposition}

\begin{proof}
Since we are only interested in the estimate \eqref{tDprop5.1}
for big $t$, we can normalize the function $\vp$ such that
$\vp(0)=1$.

Split $\R$ into subintervals $[t_n,t_{n+1})$ (possibly
adding $(-\infty,t_0]$).

Let us construct a $\tilde{\Delta}$-admissible function.
To begin with let $\gamma_0$ be the continuous and
piecewise affine function defined on each interval $[t_n,t_{n+1})$ by
\beqa
 \gamma_0:[t_n,t_{n+1})&\lra&[t_{n+1},t_{n+2}),\\
 t&\lmto&t_{n+1}+\frac{t_{n+2}-t_{n+1}}{t_{n+1}-t_n}(t-t_n).
\eeqa
This is just the affine increasing bijection from $[t_n,t_{n+1})$
onto $[t_{n+1},t_{n+2})$.
Define moreover $\gamma_1(t)=\gamma_0(t)-t$ so that
$\gamma_1(t_n)=t_{n+1}-t_n$ for every $n$.
This function is still continuous and
piecewise affine. Moreover it tends to infinity since the
sequence $(t_{n+1}-t_n)_n$ does and since it 
is bounded below on any interval $[t_n,t_{n+1})$ by the values
$\gamma(t_n)$ and $\gamma(t_{n+1})$. It is clear that we can then
bound below $\gamma_1$ by a function $\gamma$ which is concave
(one could construct such a function as a continuous piecewise
affine function with decreasing growth coefficient on each interval).

Let us check that the so obtained function $\gamma$ 
satisfies the $\tilde{\Delta}$-admissibility type condition
\eqref{tDprop5.1}. Let $t\in\R$ and suppose $t\in
[t_n,t_{n+1})$. Observe that then $\gamma_0(t)\in [t_{n+1},t_{n+2})$.
Hence
\beqa
 \vp(t+\gamma(t))\le \vp(t+\gamma_1(t))=\vp(\gamma_0(t))
 \le \vp(t_{n+2})\le 2^2\vp(t_n)\le 2^2\vp(t).
\eeqa
\end{proof}

As a consequence of the previous proposition and
Theorem \ref{thm1} we obtain

\begin{corollary}\label{cor5.1}
Let $\vp$ be as in the proposition. 
There exists a strongly
convex function $\psi$ such that
\beqa
 \Alg(\mathcal{H}_{\Psi})\subset \mult(\mH_{\Phi}),
\eeqa
where $\Psi=\psi\circ\log$.
\end{corollary}

Important examples of strongly convex functions for which the
sequence $(t_{n+1}-t_n)_n$ tends to infinity are given by
$\vp_k^{(\log)}(t)=e^{t/\log_k t}$, $k\in\N^*$, $\vp(t)=e^{t/\sqrt{\log_k(t)}}$,
and it is even possible to construct functions $\vp(t)$ that behave
on intervals $I_n$ like $\vp_n^{\log_n}$.

Let us discuss more thouroughly the case of $\vp_k^{(\log)}$.
This function defines a Hardy-Orlicz space that is very
close to $\bigcup_{p>0}H^p$ and having unbounded multipliers.
We will check that $\gamma_{k,c}(t)=c\log_k(t)$ is admissible:
\beqa
 \frac{t+c\log_k t}{\log_k(t+c\log_k t)}-\frac{t}{\log_k t}
 &=&\frac{t\log_k t+c\log^2_k t-t\log_k(t+c\log_k t)}
 {\log_k t \log_k(t+c\log_k t)}\\
 &=&\frac{c\log^2_k t -t (\log_k(t+c\log_k t)-\log_{k} t)}
 {\log_k t \log_k(t+c\log_k t)}\\
 &\le& c\frac{\log_k(t)}{\log_k(t+\log_k t)}\le c.
\eeqa
Also $\gamma_{k,c}^{-1}(t)=\exp_k(t/c)$ where $\exp_k
=\underbrace{\exp\circ\cdots\circ\exp}_{k \text{ times}}$.
So
\beqa
 \tilde{\Psi}_{k,\alpha}(t)
 &=&\vp\circ\gamma_{k,c}^{-1}\circ\log t
 =\vp(\exp_k\frac{\log t}{c})=\vp(\exp_{k-1}t^{\alpha})
 =\exp(\frac{\exp_{k-1}t^{\alpha}}{\log_k\exp_{k-1}t^{\alpha}})\\
 &=&\exp(\frac{\exp_{k-1}t^{\alpha}}{\alpha\log t}).
\eeqa
Setting also
\beqa
 \Psi_{k,\alpha}(t)=\exp_k t^{\alpha},
\eeqa
we again get
\beqa
 \bigcup_{\alpha>0}\mH_{\Psi_{k,\alpha}}
 =\bigcup_{\alpha>0}\mH_{\tilde{\Psi}_{k,\alpha}}
 \subset \mult(\mH_{\Phi_{k}}).
\eeqa

The spaces $\mH_{\Psi_{\alpha}}$ (and a fortiori the spaces
$\mH_{\Psi_{k,\alpha}}$) are extremely small, by which we mean
that they are very close to $\Hi$. This can be expressed
by the Boyd indices. 
For Orlicz spaces, \cite[Proposition 2.b.5]{LT} gives an
explicit formula allowing the computation of these indices.
It turns out that --- not very surprisingly --- 
$p_X=q_X=+\infty$ for $X=\mH_{\Psi_{k,\alpha}}$.

\section{Interpolation}\label{S4}

In this section we will consider the interpolation problem 
in 
Hardy-Orlicz spaces beyond $\bigcup_{p>0}H^p$. 

We shall begin by recalling some definitions. 
The interpolation problem we would like to consider is that
of free interpolation. 

\begin{definition}\label{def4.1}
A sequence $\Lambda=\{\lambda_n\}_n\subset \DD$ is called
a free interpolating sequence for a space of holomorphic functions
on $\DD$, $X=\Hol(\DD)$, if for every $f\in X$, and for every
sequence $(b_n)_n$ with 
\beqa
 |b_n|\le |f(\lambda_n)|,\quad n\in\N,
\eeqa
there exists a function $g\in X$ such that $g(\lambda_n)=b_n$,
$n\in\N$.

Notation: $\Lambda\in\Int_{l^{\infty}}X$.
\end{definition}

Another way of expressing that a sequence is of free interpolation
is to say that $l^{\infty}$ is contained in the
multiplier algebra of $X|\Lambda:=\{(f(\lambda_n))_n:
f\in X\}$: for every $(a_n)_n
=(f(\lambda_n))_n\in X|\Lambda$ and for every $\mu=(\mu_n)_n
\in l^{\infty}$ there is $g\in X$ such that $g(\lambda_n)=\mu_n a_n$,
$n\in\N$, i.e.\ $(\mu_na_n)_n\in X|\Lambda$.

It is clear that if we can interpolate the bounded sequences
by functions in the multiplier algebra, i.e.\ $l^{\infty}\subset
\mult(X)|\Lambda$, then $\Lambda\in  \Int_{l^{\infty}}X$.

The definition of free interpolation 
originates in the work by Vinogradov and Havin
in the 70s. It is very well adapted to the Hilbert space situation
where it can be connected to the unconditionality of a sequence
of reproducing kernels, see e.g.~\cite[Theorem C3.1.4, Theorem
C3.2.5]{Nik02} for
a general source; see also \cite{HMNT} or \cite{H06} for more motivations
for the non-Banach situation. 

Let us recall that by a famous result of L. Carleson \cite{carl}
the interpolating sequences for $\Hi$, i.e.\ the sequences
$\Lambda$ for which $\Hi|\Lambda=l^{\infty}$, are
characterized by the Carleson condition:
\bea\label{carlcond}
 \inf_{\lambda\in\Lambda}|B_{\Lambda\setminus\lambda}(\lambda)|
 =\delta>0.
\eea
Here $B_E=\prod_{\lambda\in E}b_{\lambda}$ 
is the Blaschke product associated with a discrete 
set $E\subset \DD$ (supposed to satisfy the Blaschke condition
$\sum_{\lambda\in E}(1-|\lambda|^2)<\infty$).
Recall that for $\lambda\in\DD$
\beqa
 b_{\lambda}(z)=\frac{|\lambda|}{\lambda}\frac{\lambda-z}
 {1-\overline{\lambda}z},\quad z\in\DD.
\eeqa
A sequence satisfying \eqref{carlcond} will be called a
{\it Carleson sequence}.
It is clear that for $X=\Hi$ classical interpolation and
free interpolation are the same.

The Carleson condition still characterizes interpolating
sequences (free or classical) in a large class of Hardy-Orlicz 
spaces included in the scale of $H^p$ spaces 
(see \cite{ShHSh} for $H^p$, $p\ge 1$; \cite{Kab} for $H^p$, $p<1$
and \cite{Har} for more general Hardy-Orlicz spaces included in
the scale of classical Hardy spaces $H^p$).

The situation is intrinsically different in spaces close to
the Nevanlinna and Smirnov classes. Here interpolating sequences
are characterized by the existence of harmonic majorants of 
the function $\vp_{\Lambda}$ defined by $\vp_{\Lambda}(\lambda)=
\log\frac{1}{|B_{\lambda}(\lambda)|}$ when $\lambda\in\Lambda$
and $\vp_{\Lambda}=0$ otherwise. See \cite{HMNT} for precise
results in the Nevanlinna and Smirnov classes and
\cite{H06} for big Hardy-Orlicz spaces where $\mult(\mH_{\Phi})
=\mH_{\Phi}$.

Of course a big gap remains between big Hardy-Orlicz spaces
considered in \cite{H06}
and $\bigcup_{p>0} H^p$. In particular 
an intriguing question is
to know whether there are Hardy-Orlicz spaces beyond $\bigcup_{p>0}H^p$
where the Carleson condition still characterizes the interpolating
sequences. In the light of Theorem \ref{thm4.3}, this question
is still more exciting since there are very large Hardy-Orlicz spaces
for which the multipliers reduce to $\Hi$. 
Here we will give examples of Hardy-Orlicz spaces which 
are close to the union $\bigcup_{p>0}H^p$ and which have
free interpolating sequences which are not Carleson. 

We will consider
the problem through the multiplier algebra of
the Hardy-Orlicz space under consideration. 
As already explained, the idea is 
to solve the interpolation problem: find $\Lambda=\{\lambda_n\}_n
\subset\DD$ such that 
\beqa
 l^{\infty}\subset \mult(\mH_{\Phi})|\Lambda.
\eeqa
Then $\Lambda$ is a free interpolating sequence for $\mH_{\Phi}$, and
in our context we would like that $\Lambda$ is
not a Carleson sequence.

The situation we will consider here is that of a Hardy-Orlicz
space $\mH_{\Phi}$ 
the multiplier algebra of which contains $\mH_{\Psi}$ where
$\Psi=\psi\circ\log$ and
$\psi:\R\to [0,\infty)$
is a strongly convex function.
Examples of such a situation can be deduced from Corollary 
\ref{cor5.1}.
In such a situation $\mH_{\Psi}$ contains not only $\Hi$ but also --- and
this will be important for us --- unbounded functions such as for example
the outer function $g$ with $|g|=\Psi^{-1}\circ v_1$ 
a.e.\ on $\T$, where $v_1(t)=\frac{\D 1}{\D t\log^{1+\eps}(1/t)}$
and $\eps>0$.

Let $M:=\alg(\mH_{\Psi})$ which is inluded in $\mult (\mH_{\Phi})$. 

We need two simple properties on $M$. Recall from \eqref{outer}
that for a function
$f$ in the Smirnov class we have written $[f]$ for its outer part.
We will use more generally this notation for the 
outer function associated with a measurable 
function $f$ on $\T$ with $\log |f|\in L^1$.

\begin{lemma}\label{lem6.2}
If $f\in M$ then there exists $n\in\N^*$ such that 
$[f]^{1/n}\in \mH_{\Psi}$
\end{lemma}

\begin{proof}
We begin by checking the result for products and sums
of functions in the generator $\mH_{\Psi}$ of $M$.
Observe first that if $f_1,f_2\in\mH_{\Psi}$, then
$[w]\in\mH_{\Psi}$ where $w:=\max(|f_1|,|f_2|)$ (just split the
integral $\int_{\T}\Psi(|w|)dm$ into two parts where $|f_1|$ 
(respectively $|f_2|$) has bigger modulus).
So, if $f=f_1f_2$ then $|f|\le w^2$ and $[f]^{1/2}\in\mH_{\Psi}$.
By a simple induction this holds for finite products.

Of course, $[f_1+f_2]\in\mH_{\Psi}$ whenever $f_1,f_2
\in\mH_{\Psi}$, and this extends obviously to
finite sum of functions in $\mH_{\Psi}$. 

Let us now look whether the property holds for products and
sums of functions in $M$.
If $f_1,f_2\in M$ 
with $[f_1]^{1/n}\in\mH_{\Psi}$, $[f_2]^{1/k}\in\mH_{\Psi}$
then
$[w]^{1/N}\in \mH_{\Psi}$, where $w:=\max(|f_1|,|f_2|)$ 
and $N=\max(n,k)$
(just split the
integral $\int_{\T}\Psi(|w|^{1/N})dm$ into two parts where $|f_1|$ 
(respectively $|f_2|$) has bigger modulus; the case when
$|f_1|\le 1$ or $|f_2|\le 1$ is of no relevance here).  
Hence, if
$f=f_1f_2$ 
then
$|f|\le |[w]|^{2}$, 
and $\int_{\T}\Psi(|f|^{1/(2N)})dm\le
\int_{\T}\Psi(|w|^{1/N})dm<\infty$, i.e.\
$[f]^{1/(2N)}\in\mH_{\Psi}$. By a simple induction
this also holds for finite products.

For sums of functions in $M$,
let $f_1,f_2,w,N$ as above. In particular $[w]^{1/N}\in\mH_{\Psi}$.
If now $f=f_1+f_2$, then $|f|\le 2w$
so that $|[f]^{1/N}|\le |2w|^{1/N}$ from where we deduce
that $[f]^{1/N}\in \mH_{\Psi}$.
By a simple induction this generalizes to finite sums.

Since the property of the lemma is true for functions
in $\mH_{\Psi}$ and it is conserved by finite sums and
products of functions in $M=\alg(\mH_{\Psi})$ 
it holds for the algebra generated by $\mH_{\Psi}$.
\end{proof}

A simple consequence is the following.

\begin{corollary}\label{cor6.3}
If $f\in M$, then $[\max(1,|f|)]\in M$.
\end{corollary}

\begin{proof}
From the lemma we obtain that $[f]^{1/n}\in\mH_{\Psi}$ for
a convenient $n\in\N^*$. Then clearly 
$h:=[\max(1,|f|)]^{1/n}=[\max(1,|f|^{1/n})]\in\mH_{\Psi}$.
Hence $[\max(1,|f|)]=h^n\in M=\Alg \mH_{\Psi}$.
\end{proof}

We can add another consequence of Lemma \ref{lem6.2}.

\begin{corollary}\label{cor6.4}
We have $\Alg(\mH_{\Psi})= \bigcup_{n\in\N} \mH_{\Psi_n}$ where
$\Psi_n(t)=\Psi(t^{1/n})$.
\end{corollary}

It should be noted that $\Psi_n$ is not necessarily 
convex, but $\psi_n(t):=\Psi_n\circ\exp(t)=\psi(t/n)$
is still stongly convex in the terminology of \cite{Ru} so
that we still 
can define the corresponding Hardy-Orlicz classes
(which are not necessarily vector spaces). 
In the case $\Psi_{1,\alpha}(t)=e^{t^{\alpha}}$,
which defines a Hardy-Orlicz space
contained in the multiplier algebra of $\mH_{\Phi_{1}^{(\log)}}$ 
($\Phi_1^{(\log)}(t)=e^{\log t/\log\log t}$ for $t$ sufficiently big),
$(\Psi_{1,\alpha})_n$ will be convex (we have taken the notation
from the end of Subsection \ref{sectexam}).

Like in \cite{DSh} 
our example of a free interpolating sequence will be
constructed as a non separated union of two Carleson sequences
(this is different to \cite{TW} where Carleson's method
is used to interpolate $l^q$-sequences by $H^p$-functions).
In order to do that we will use the results of \cite{Ha}
based on the so-called (C)-stability.

Let us recall the definition of (C)-stability (see \cite{Ha}).

\begin{definition}
Let $X\subset\Hol(\DD)$. If there exists $\delta_0\in (0,1)$
such that for every pair of Carleson sequences
$\Lambda=\{\lambda_n\}_n\subset\DD$ and 
$\tilde{\Lambda}=\{\tilde{\lambda}_n\}_n\subset\DD$ with
\beqa
 \sup_n |b_{\lambda_n}(\tilde{\lambda}_n)|=\delta<\delta_0
\eeqa
we have
\beqa
 X|\Lambda=X|\tilde{\Lambda}
\eeqa
then $X$ is called (C)-stable.
\end{definition}

Since $\Hi\subset \mH_{\Psi}$ and $M$ is an algebra containing
$\mH_{\Psi}$ we also have
$\Hi \subset M=\mult(M)$ 
which in particular implies that a Carleson sequence
is a free interpolating sequence for 
$M$.

\begin{proposition}
The space $M$ is $(C)$-stable.
\end{proposition}

\begin{proof}
Pick $f\in M$, 
and let $\Lambda$, $\tilde{\Lambda}$ as in the definition.
Set $a_n=f(\lambda_n)$.
We have to verify that $\{a_n\}_n\in M|\tilde{\Lambda}$.
Put $w=\max(1,|f|)$ a.e.\ $\T$.
By Corollary \ref{cor6.3}, $F:=[w]\in M$.
It is clear that $|a_n|\le |A_n|$ where $A_n=F(\lambda_n)$.
Note that $\log |F|$ is by construction a positive harmonic 
function, and so by Harnack's inequality 
there is a constant $c>1$ such that 
\beqa
 |F(\tilde{\lambda}_n)|^{1/c}\le |F(\lambda_n)|\le
 |F(\tilde{\lambda}_n)|^c,\quad n\in\N.
\eeqa
So 
\beqa
 |a_n|\le |F^c(\tilde{\lambda}_n)| 
\eeqa
Let $n$ be a natural number bigger than $c$. Then $|F^c|\le |F^n|$
and $F^n\in M$ by Lemma \ref{lem6.2}.

Since $\tilde{\Lambda}$ is a Carleson sequence by assumption,
and so a free interpolating sequence for $M$, there exists
a function $g\in M$, such that
\beqa
 g(\tilde{\lambda}_n)=a_n,\quad n\in \N.
\eeqa
Hence $M|\Lambda\subset M|\tilde{\Lambda}$. Since the problem
is symmetric, we also have the reverse inclusion, and $M$
is (C)-stable.
\end{proof}

We will now examine the trace of $\mH_{\Psi}^*$. For our purpose
it will be sufficient to know the restriction 
$\mH_{\Psi}^*|\Lambda$ when $\Lambda$ is a Carleson sequence.
For this we will use the Jones-Vinogradov interpolation operator
(see e.g.\ \cite[Vol.2, pp.179-180]{Nik02}),
which with a sequence $a=\{a_n\}_n$ associates a holomorphic
function
\beqa
 Ta(z)=\sum_{n\in\N}a_n f_n(z),\quad z\in \DD.
\eeqa
The exact form of the functions $f_n$ is not very
interesting for our discussion here (we refer the reader to the
above cited monograph, or to \cite{Ha}). The family $(f_n)$ is
of course a Beurling-type family, by which we mean that
$f_n(\lambda_k)=\delta_{nk}$ and
\beqa
 \sup_{z\in\DD}\sum_{n\in\N}|f(z)|<\infty.
\eeqa
The operator $T$ is continuous from $l^1(1-|\lambda|^2)=
\{a=(a_n)_n:\sum_{n\in\N}
\|a\|_{l^1(1-|\lambda|^2)}:=(1-|\lambda_n|^2)|a_n|<\infty\}$ 
to $H^1$ and from $l^{\infty}$ to $H^{\infty}$
(see the above cited monograph). 
These results suggest the use of interpolation 
between Banach space (lattices). 
In order to do this we will adapt a Calder\'on
interpolation theorem for rearrangement invariant subspaces
(see e.g.\ \cite[Theorem 2.a.10]{LT}) to our situation.

The space
\beqa
 l_{\Psi}^*(1-|\lambda_n|^2):=\{a=(a_n)_n:\exists C>0,
 \sum_{n\in\N}(1-|\lambda_n|^2)
 \Psi\left(\frac{|a_n|}{C}\right) <\infty\},
\eeqa
equipped with the usual norm $\|\cdot\|_{\Phi}$
is a Banach space. 

\begin{proposition}\label{prop6.6}
Let $\Lambda\in (C)$. The operator $T$ is continuous from
$l^*_{\Psi}(1-|\lambda_n|^2)$ to $\mH_{\Psi}^*$.
\end{proposition}

Consequently, if $\Lambda\in (C)$ then
\beqa
 l_{\Psi}^*(1-|\lambda_n|^2)\subset \mH_{\Psi}^*|\Lambda.
\eeqa

\begin{proof}
We have already introduced the distribution function and the
decreasing rearrangement of a function defined on a measure space.
We now have to consider these notions in the sequence
space $l_{\Psi}^*(1-|\lambda_n|^2)$ (the underlying measure space
being $\N$ with the measure $\mu=\sum_{n\in\N}(1-|\lambda_n|^2)\delta_n$)
and in the Lebesgue space $L_{\Psi}^*$.

We start with a sequence 
$a\in l_{\Psi}^*(1-|\lambda_n|^2)$.
Repeating the arguments of the proof of Calder\'on's theorem
given in \cite[Theorem 2.a.10]{LT}, we set for our sequence
$a$ and an $s\in [0,L]$, $L:=\sum_{n\in\N}(1-|\lambda_n|^2)<\infty$,
\beqa
 b_n^s=\left\{
 \begin{array}{ll}
 (|a_n|-a^*(s))\frac{a_n}{|a_n|} & \text{ if }|a_n| > a^*(s)\\
 0 & \text{ if }|a_n|\le a^*(s)
 \end{array}
 \right.
\eeqa
and $c_n^s=a_n-b_n^s$. Clearly, $\|c^s\|_{l^{\infty}}\le a^*(s)$.

\begin{center}
\begin{picture}(0,0)%
\includegraphics{decrrearr.pstex}%
\end{picture}%
\setlength{\unitlength}{1973sp}%
\begingroup\makeatletter\ifx\SetFigFont\undefined%
\gdef\SetFigFont#1#2#3#4#5{%
  \reset@font\fontsize{#1}{#2pt}%
  \fontfamily{#3}\fontseries{#4}\fontshape{#5}%
  \selectfont}%
\fi\endgroup%
\begin{picture}(11937,7875)(601,-8224)
\put(5926,-7636){\makebox(0,0)[lb]{\smash{{\SetFigFont{8}{9.6}{\familydefault}{\mddefault}{\updefault}{\color[rgb]{0,0,0}$s$}%
}}}}
\put(11626,-7561){\makebox(0,0)[lb]{\smash{{\SetFigFont{8}{9.6}{\rmdefault}{\mddefault}{\updefault}{\color[rgb]{0,0,0}$L$}%
}}}}
\put(601,-4636){\makebox(0,0)[lb]{\smash{{\SetFigFont{8}{9.6}{\familydefault}{\mddefault}{\updefault}{\color[rgb]{0,0,0}$a^*(s)=a_{n_s}$}%
}}}}
\put(5326,-8161){\makebox(0,0)[lb]{\smash{{\SetFigFont{8}{9.6}{\rmdefault}{\mddefault}{\updefault}{\color[rgb]{0,0,0}$1-|\lambda_{n_s}|^2$}%
}}}}
\put(4951,-7636){\makebox(0,0)[lb]{\smash{{\SetFigFont{8}{9.6}{\familydefault}{\mddefault}{\updefault}{\color[rgb]{0,0,0}$s_{n_s}$}%
}}}}
\end{picture}%
\\
Figure 2: Decreasing rearrangement of $(a_n)_n$
\end{center}

Also, the norm of $b^s$ in $l^1(1-|\lambda_n|^2)$ corresponding
to the hatched region in Figure 1 can be computed as follows
\beqa
 \|b^s\|_{l^1(1-|\lambda_n|^2)}
 &=&\sum_{n\in\N}(1-|\lambda_n|^2)|b_n^s|
 =\sum_{n:|a_n|> a^*(s)}(1-|\lambda_n|^2)(|a_n|-a^*(s))\\
 &=&\sum_{n:|a_n|> a^*(s)}(1-|\lambda_n|^2)|a_n|
 -a^*(s)\sum_{n:|a_n|> a^*(s)}(1-|\lambda_n|^2)\\
 &=&\int_0^sa^*(t)dt-sa^*(s)
\eeqa
(here $n_s$ is an integer with $a^*(s)=a_{n_s}$
and $s_{n_s}=\sum_{n:|a_n|> a^*(s)}(1-|\lambda_n|^2)$).

Now, $T$ is linear, 
and by a well known estimate on decreasing rearrangements
$(Ta)^*(s+s)\le (Tb^s)^*(s)+(Tc^s)^*(s)$.
Hence, we obtain as in the proof of Calder\'on's theorem
\beqa
 \int_0^s(Ta)^*(t)dt
 &=& 2\int_0^{s/2}(Ta)^*(2s)ds
 \le 2\int_0^{s/2}(Tb^s)^*(s)ds+2\int_0^{s/2}(Tc^s)^*(s)ds\\
 &\le& 2\int_0^{s/2}(Tb^s)^*(s)ds+2\int_0^{s/2}\|Tc^s\|_{\infty}ds\\
 &\le& 2\|Tb^s\|_1+s\|Tc^s\|_{\infty}\\
 &\le&2\max(\|T\|_{l^1(1-|\lambda_n|^2)\to
   H^1},\|T\|_{l^{\infty}\to\Hi}) (\|b^s\|_{l^1(1-|\lambda_n|^2)}
 +s a^*(s))\\
 &=&2\max(\|T\|_{l^1(1-|\lambda_n|^2)\to
   H^1},\|T\|_{l^{\infty}\to\Hi})\int_0^sa^*(t)dt
\eeqa
The function $g$ 
defined by
\beqa
 g(e^{2\pi it})=a^*(Lt), \quad t\in (0,1],
\eeqa
is in $L^*_{\Psi}$ (recall that $L$ was the Blaschke sum
$\sum(1-|\lambda_n|^2)$ corresponding to the measure
$\mu(\N)$), and the above inequality becomes
\bea\label{riestim}
  \int_0^s(Ta)^*(t)dt\le c\int_0^s g^*(e^{2\pi it})dt,\quad \forall s\in (0,1]
\eea
(here $c$ is a suitable constant).

Now, $L_{\Psi}^*(\T)$ is a rearrangement invariant space
(see \cite[p.120]{LT}) and so, by \cite[Proposition 2.a.8]{LT},
we deduce from \eqref{riestim} that $Ta\in L_{\Psi_d}^*(\T)$
and that $\|Ta\|_{L_{\Psi}^*}\le c_1 \|g\|_{L_{\Psi}^*}
\le {c_2} \|a\|_{l_{\Psi}^*}$. 
This achieves the proof
\end{proof}

We should mention that we do not know whether $\mH_{\Phi}^*|\Lambda$
embeds into $l^{*}_{\Psi}(1-|\lambda_n|^2)$, and for this reason it is
not clear if $\mH_{\Phi}^*$ is $(C)$-stable. This explains why we
pass through $M$ which we know to be $(C)$-stable.

Let us now turn to the construction of an interpolating
sequence for $\mH_{\Phi}$ not satisfying the Carleson
condition. As already mentioned, for that it is sufficient 
to construct a sequence $\Lambda$ which is not Carleson
yet $M|\Lambda$ contains $l^{\infty}$. We will use
Theorem 1.4 of \cite{Ha} (the idea of course goes back
to \cite{DSh}.

\begin{proposition}
There exists a sequence $\Lambda\not\in (C)$ such that $M|\Lambda\supset
l^{\infty}$.
\end{proposition}

\begin{proof}
Let $\Lambda_j=\{\lambda_{n,j}\}\subset \DD$ 
be Carleson sequences, $j=1,2$, such that 
$|b_{\lambda_{n,1}}(\lambda_{n,2})|\le \delta_0$ and such
that
\beqa
 \lim_{n\to\infty}|b_{\lambda_{n,1}}(\lambda_{n,2})|=0.
\eeqa
The latter condition guarantees that $\Lambda=\Lambda_1
\cup\Lambda_2$ is not Carleson. 
The condition on the speed of convergence to zero
of $(|b_{\lambda_{n,1}}(\lambda_{n,2})|)_n$ will be fixed later.

Let $M(\Lambda_1)=M|\Lambda_1$ ($=M|\Lambda_2$), and
set 
\beqa
 M_2(\Lambda):=\{(a_{n,i})_{n\in\N,i=1,2}:
 (a_{n,1})_n\in 
 M(\Lambda_1), \left(\frac{a_{n,1}-a_{n,2}}
 {b_{\lambda_{n,2}}(\lambda_{n,1})}\right)_n\in 
 M(\Lambda_1)\},
\eeqa
which is a kind of inductive limit
of first order discrete Sobolev-Orlicz spaces.
Since $M$ is (C)-stable, we deduce from
\cite[Theorem 1.4]{Ha} that
\beqa
 M_2(\Lambda)\subset M|\Lambda:=\{(f(\lambda_{n,i}))_{n\in\N,i=1,2}:
 f\in M\}
\eeqa
(the careful reader might have observed that we only use one
half of that theorem, but this is sufficient for our
purpose here since we are only interested in 
one inclusion). 
Set also
\beqa
 l_{\Psi,2}^*(1-|\lambda_{n,i}|^2):=\{(a_{n,i})_{n\in\N,i=1,2}:
 (a_{n,1})_n\in l_{\Psi}^*(1-|\lambda_{n,1}|^2), \left(\frac{a_{n,1}-a_{n,2}}
 {b_{\lambda_{n,2}}(\lambda_{n,1})}\right)_n\in l_{\Psi}^*
 (1-|\lambda_{n,1}|^2)\},
\eeqa
and analogously $l_{\Psi,2}(1-|\lambda_{n,i}|^2)$ by omitting
the stars everywhere in the previous definition.
By Proposition \ref{prop6.6}, 
$l_{\Psi}^*(1-|\lambda_{n,1}|^2)\subset \mH_{\Psi}^*|\Lambda
\subset M(\Lambda_1)$, and
so $l^*_{\Psi,2}(1-|\lambda_{n,i}|^2)\subset M_2(\Lambda)$.
In particular we can interpolate every sequence
$(a_{n,i})_{n\in\N,i=1,2}$ with
\beqa
 \sum_{n\in\N} (1-|\lambda_{n,1}|^2)\Psi\left(|a_{n,1}|+
 \left|\frac{a_{n,2}-a_{n,1}}{b_{\lambda_{n,1}}(\lambda_{n,2})}
 \right|\right) <\infty
\eeqa
by a function in $M$. 
Now, since $\Lambda_1$ is a Blaschke sequence, there exists
an increasing sequence $(\gamma_n)_n$ of positive elements
tending to infinity and such that
\beqa
 \sum_{n\in\N}(1-|\lambda_{n,1}|^2)\gamma_n<\infty.
\eeqa
Choosing $\Lambda_2$ such that
$(|b_{\lambda_{n,1}}(\lambda_{n,2})|)_n$ goes to zero and
\beqa
 |b_{\lambda_{n,1}}(\lambda_{n,2})|\ge
 \frac{2}{\Psi^{-1}(\gamma_n)-1},
\eeqa
we obtain for every $a\in l^{\infty}$ with $\|a\|_{\infty}\le 1$,
\beqa
 \sum_{n\in\N} (1-|\lambda_{n,1}|^2)\Psi\left(|a_{n,1}|+
 \left|\frac{a_{n,2}-a_{n,1}}{b_{\lambda_{n,1}}(\lambda_{n,2})}
 \right|\right) 
 &\le& \sum_{n\in\N} (1-|\lambda_{n,1}|^2)
 \Psi\left(1+\frac{2}{|b_{\lambda_{n,1}}(\lambda_{n,2})|}\right)\\
 &\le& \sum_{n\in\N} (1-|\lambda_{n,1}|^2)\gamma_n<\infty.
\eeqa
Hence, the unit ball of $l^{\infty}$ is in 
$M|\Lambda$,
and so also the whole space $l^{\infty}$ since $M$ 
is a vector space.
We are done.
\end{proof}

As a consequence we obtain the following result.

\begin{corollary}\label{cor6.9}
Let $\vp$ be a strongly convex function, and $\Psi$ a strictly
increasing, convex, unbounded function such that
\beqa
 \mH_{\Psi}\subset\mult (\mH_{\vp\circ\log}).
\eeqa
Then there exists $\Lambda\not\in (C)$ such that $\Lambda\in
\Int_{l^{\infty}}\mH_{\vp\circ\log}$.
\end{corollary}

\providecommand{\bysame}{\leavevmode\hbox to3em{\hrulefill}\thinspace}
\providecommand{\MR}{\relax\ifhmode\unskip\space\fi MR }

\end{document}